\documentclass[onefignum,onetabnum]{siamart190516}

\usepackage{nicefrac}
\usepackage{algpseudocode}
\usepackage{algorithm}
\usepackage{subcaption}
\usepackage{xcolor}

\usepackage{array, makecell}
\usepackage{amsfonts}
\usepackage{graphicx}
\usepackage{epstopdf}
\usepackage{subcaption}
\usepackage{adjustbox}
\ifpdf
  \DeclareGraphicsExtensions{.eps,.pdf,.png,.jpg}
\else
  \DeclareGraphicsExtensions{.eps}
\fi

\newsiamremark{remark}{Remark}
\newsiamremark{hypothesis}{Hypothesis}
\crefname{hypothesis}{Hypothesis}{Hypotheses}
\newsiamthm{claim}{Claim}

\newcommand{\bB}{\bar{B}}
\newcommand{\bL}{\bar{L}}
\newcommand{\eps}{\varepsilon}
\newcommand{\Kry}{\mathcal{K}}
\newcommand{\lamk}{_{\lambda,k}}
\newcommand{\mcE}{\mathcal{E}}
\newcommand{\mcF}{\mathcal{F}}
\newcommand{\mcR}{\mathcal{R}}

\newcommand{\R}{\mathbb{R}}
\newcommand{\RRE}{\mbox{RRE}}

\newcommand{\true}{{\text{\footnotesize{{true}}}}}

\newcommand{\whA}{\widehat{A}}
\newcommand{\whb}{\widehat{b}}
\newcommand{\whE}{\widehat{E}}
\newcommand{\whmcE}{\widehat{\mathcal{E}}}
\newcommand{\whH}{\widehat{H}}
\newcommand{\whr}{\widehat{r}}

\newcommand{\wtAlam}{\widetilde{A}_{\lambda}}
\newcommand{\wtE}{\widetilde{E}}
\newcommand{\wteps}{\widetilde{\eps}}
\newcommand{\wtmcE}{\widetilde{\mathcal{E}}}
\newcommand{\wtr}{\widetilde{r}}
\newcommand{\wts}{\breve{s}}

\headers{Regularization by inexact Krylov methods}{S. Gazzola and M. Sabat\'{e} Landman}

\title{Regularization by inexact Krylov methods\\
with applications to 
blind 
deblurring
\thanks{Submitted to the editors DATE.
\funding{{S. Gazzola is partially funded by EPSRC, under grant EP/T001593/1. M. Sabat\'e Landman is supported by a scholarship from the EPSRC Centre for Doctoral Training in Statistical Applied Mathematics at Bath (SAMBa), under project EP/L015684/1.}}}}

\author{Silvia Gazzola \thanks{Department of Mathematical Sciences, University of Bath, United Kingdom 
(\email{S.Gazzola@bath.ac.uk}, \url{https://people.bath.ac.uk/sg968/})}
\and Malena Sabat\'{e} Landman\thanks{Department of Mathematical Sciences, University of Bath, United Kingdom 
(\email{M.Sabate.Landman@bath.ac.uk}, \url{https://people.bath.ac.uk/msl39/})}}

\begin{document}

\maketitle

\begin{abstract}
This paper is concerned with the regularization of large-scale discrete inverse problems by means of inexact Krylov methods. Specifically, we derive two new inexact Krylov 
methods that can be efficiently applied to unregularized or Tikhonov-regularized least squares problems, and 
we study their theoretical properties, including links with their exact counterparts and 
strategies to monitor the amount of inexactness. 
We then apply the new methods to separable nonlinear inverse problems arising in 
blind deblurring. 
In this setting inexactness stems from the uncertainty in the parameters defining the blur, which may be 
recovered using a variable projection method leading to an inner-outer iteration scheme (i.e., one cycle of inner iterations is performed to solve one linear deblurring subproblem for any intermediate 
values of the blurring parameters computed by a nonlinear least squares solver). 
The new inexact solvers can naturally handle varying inexact blurring parameters while solving the linear deblurring subproblems, allowing for a much reduced number of total iterations and substantial computational savings with respect to their exact counterparts. 
\end{abstract}



\section{Introduction}\label{sec:intro}
Linear discrete inverse problems of the form 
\begin{equation}\label{eq:linsys}
\min_{x\in\R^n}\|Ax -b\|\,,\quad\mbox{where}\quad b= b_{\true} + e =Ax_{\true} + e\,,\; \quad  \frac{\|e\|}{\|b_{\true}\|}\ll 1\,,
\end{equation}
and where $e$ is an unknown perturbation 
affecting the data $b$, arise in many engineering and scientific applications; see \cite{Chung2015, JulianneJim, book1, PCH10}. Here and in the following, $\|\cdot\|$ 
denotes the vector 2-norm and the induced matrix 2-norm. We assume that the discretized forward operator $A\in\R^{m\times n}$, with $m\geq n$, 
has full column rank with singular values quickly decaying and clustering at zero, so that $A$ is ill-conditioned. Although, under these assumptions, problem (\ref{eq:linsys}) has a unique solution, the ill-conditioning of $A$ and the presence of the noise $e$ in $b$ prompt the use of some regularization to recover a meaningful approximation of $x_{\true}$. 

In this paper we consider Tikhonov regularization, which computes
\begin{equation}\label{tikh}
x_{\true}\approx x_{\lambda} = \arg\min_{x\in\R^n}\:
\|Ax-b\|^2\,+\,\lambda^2\|x-x_0\|^2\,,
\end{equation}
where an initial estimate $x_0$ for 
$x_{\true}$ is included. Basically, Tikhonov regularization replaces the original least squares (LS) problem (\ref{eq:linsys}) by a \emph{penalized LS problem}, where a `small' value of the regularization term $\|x-x_0\|^2$, weighted by a regularization parameter $\lambda>0$, is enforced. 
Prescribing a suitable value of $\lambda$ is crucial to compute solutions that are neither under-regularized ($\lambda$ too small) nor over-regularized ($\lambda$ too big). 
Equivalently, with a simple change of variable, problem (\ref{tikh}) can be reformulated as
\begin{equation}\label{tikhx0}
z_{\lambda} = \arg\min_{z\in\R^n}\:
\|Az-r_0\|^2\,+\,\lambda^2\|z\|^2\,,
\quad\mbox{where}\quad
\begin{array}{lcl}
r_0 &=& b - Ax_0\\
x_{\lambda}&=&x_0+z_{\lambda}
\end{array}.
\end{equation}
Although (\ref{tikhx0}) has a closed-form solution, 
when dealing with a large-scale and unstructured $A$, and without prior knowledge of a suitable value of $\lambda$, computing a good 
$z_\lambda$ would potentially involve repeatedly applying a (matrix-free) iterative solver for LS problems, one for each considered value of $\lambda$. In this framework, many Krylov methods are successfully applied to either 
(\ref{eq:linsys}) (i.e., as stand-alone solvers that regularize by early termination of the iterations), or (\ref{tikhx0}) (in a so-called \emph{hybrid} fashion, i.e., by combining projection onto Krylov subspaces of increasing dimensions and Tikhonov regularization, with the possibility of efficiently and adaptively choosing $\lambda$ as the iterations progress); see, for instance, \cite{CMRS00, Chung2015, survey, newsurvey} and the references therein. Among the Krylov methods routinely used for regularization, the mathematically equivalent LSQR and CGLS methods are arguably among the the most popular ones, as their theoretical properties and practical performance are generally well-understood; see \cite{Hanke01, HansenMin}. 

In many situations the (discretized) forward operator $A$ is not fully known and, 
if we assume that $A$ depends on a few unknown parameters $y\in\R^p$, with $p\ll n$, then these should be recovered replacing (\ref{eq:linsys}) by a problem of the form  
\begin{equation}\label{eq:sepnonlinsys}
\min_{x\in\R^n, y\in\R^p}\|A(y)x -b\|\,,\quad\mbox{where}\quad b= b_{\true} + e =A(y_\true)x_{\true} + e\,.
\end{equation}
This simple generalisation of (\ref{eq:linsys}) results in a much more difficult problem to solve: indeed, problem (\ref{eq:sepnonlinsys}) is jointly nonlinear and nonconvex in $(x,y)$, so that, in particular, it may not have a unique solution. Moreover, since (\ref{eq:sepnonlinsys}) is ill-posed, one should apply some regularization and consider, for instance, 
\begin{equation}
z_\lambda = \arg\min_{z\in\R^n, y\in\R^p}\:
\|
A(y)z - r_0\|^2 + 
\lambda^2\|z\|^2,
\quad\mbox{where}\quad
\begin{array}{lcl}
r_0 &=& b - A(y)x_0\\
x_{\lambda}&=&x_0+z_{\lambda}
\end{array}.
\label{eqtikhy}
\end{equation}
The above problem generalizes (\ref{tikhx0}); note that Tikhonov regularization is 
applied to $z$ 
only, while regularization on $y$ is implicitly enforced by assuming $p\ll n$ (i.e., by considering a so-called reduced parameter space). 
Problems like (\ref{eq:sepnonlinsys}), (\ref{eqtikhy}) arise in a variety of signal and image processing tasks, such as instrument calibration \cite{JulianneJim, GolubVP} and super-resolution \cite{superres}, just to name a few. In this paper we are particularly interested in spatially invariant blind deblurring 
\cite{Chung2015}, where $b$ is a blurred and noisy image (reshaped as a vector) and $A(y)$ encodes information about a parametric blur (i.e., defined by the unknown parameters $y_\true$) that corrupts every pixel of the unperturbed unknown image $x_\true$ (reshaped as a vector). To mitigate the complexity of problem (\ref{eq:sepnonlinsys}) one can take advantage of separability, as the objective function in (\ref{eq:sepnonlinsys}) is linear in 
$x$. For instance, the variable projection method \cite{GolubVP} applied to (\ref{eq:sepnonlinsys}) implicitly eliminates $x$, obtaining a reduced cost functional that depends only on $y$ and that is optimized using a nonlinear LS solver. Note that, 
in a large-scale setting, to recover $x$ one should still employ an iterative linear LS solver; i.e., one should adopt the strategies mentioned above for problems (\ref{eq:linsys}), (\ref{tikhx0}). 
In particular, \cite{JulianneJim} proposes a very efficient inner-outer iteration scheme that, while computing $x$ for a fixed $y$ using a LSQR-based hybrid method, determines a suitable regularization parameter $\lambda$ on the fly; $y$ is updated using a Gauss–Newton method. 

The main goals of this paper are to introduce inexact Krylov methods for regularizing inverse problems, and to apply them to 
compute a solution of 
(\ref{eq:sepnonlinsys}), (\ref{eqtikhy}). Indeed, our investigation is prompted by the fact that, when exploiting the separability of the objective function in (\ref{eq:sepnonlinsys}), (\ref{eqtikhy}) and applying a traditional Krylov method to recover $x\in\R^n$ for a fixed $y\neq y_{\true}\in\R^p$, one is actually 
using an inexact Krylov method, in that one is generating the approximation subspace for 
$x$ by performing matrix-vector products with inexact $A(y)$ and $A^T(y)$ rather than the exact $A(y_\true)$ and $A^T(y_\true)$, respectively. 

Starting from a reformulation of the so-called inexact Golub-Kahan decomposition \cite{inexactGKB}, we derive two new inexact Krylov solvers for LS problems (\ref{eq:linsys}) 
that can be regarded as the inexact counterparts of LSQR and CGLS, respectively; also, we study bounds for the amount of inexactness, 
so that exact and inexact solvers have comparable performances. We then explain how inexact LSQR and CGLS can be used in a hybrid fashion, i.e., to solve a Tikhonov-regularized LS problem (\ref{tikhx0}). Exact hybrid methods are underpinned by a solid theory guaranteeing that first 
regularizing and then projecting (\ref{tikhx0}) 
is equivalent to first projecting (\ref{eq:linsys}) and then 
regularizing; moreover, an adaptive choice of the regularization parameter is possible thanks to the shift invariance of Krylov subspaces. When deriving inexact hybrid methods, we study conditions under which the new solvers keep enjoying such properties. Finally, focussing on the variable projection method applied to (\ref{eq:sepnonlinsys}), (\ref{eqtikhy}), we explain how the new inexact Krylov methods potentially allow updates of the parameters $y$ as soon as a new approximation of $x$ is available, i.e., at each iteration of the inexact linear solver, rather than after a suitable number of exact iterations 
(as prescribed by the inner-outer schemes mentioned above). As far as specific conditions bounding the amount of inexactness are satisfied, this strategy will result in great computational savings. 

Although we are aware of some research 
about the use of an inexact $A^T$ in iterative regularization methods for 
linear inverse problems in 
image deblurring \cite{Donatelli2006} and computed tomography \cite{Elfving}, 
to the best of our knowledge the use of inexact Krylov methods for regularization and, in particular, for solving separable nonlinear inverse problems (\ref{eq:sepnonlinsys}), (\ref{eqtikhy}), is novel. 
Inexact Krylov methods are however 
ubiquitous in a number of numerical linear algebra tasks, typically involving well-conditioned matrices: 
we refer to \cite{inexact1} for possible applications and a literature review, as well as for a comprehensive theoretical treatment and understanding of inexact Krylov methods. 
Even if the investigations in \cite{inexact1} may appear similar to some of the ones proposed here, our approach is significantly different in that: 
(i) we 
consider solvers based on the inexact Golub-Kahan 
decomposition rather than on the inexact Arnoldi algorithm; 
(ii) since $A$ and $A(y)$ are ill-conditioned, assuring that the exact and the inexact residuals are close 
is not enough to guarantee that the exact and inexact solutions are close, and alternative conditions should be derived; 
(iii) when considering (\ref{eq:sepnonlinsys}), (\ref{eqtikhy}), the exact parameters $y_\true$ are unknown and, as a consequence, the amount of inexactness in applying $A(y)$ and $A^T(y)$ is unknown and some heuristics should be adopted to monitor it. Moreover, although \cite{inexactbis} already considers inexact versions of the CG method based on either three-term recurrences or an inexact version of the Lanczos process, the new inexact CGLS method cannot be regarded as a straightforward generalization of inexact CG because of the challenges in handling inexactness in both $A$ and $A^T$. 

This paper is organized as follows. Section \ref{sec:exact} reviews some background material about Krylov methods based on Golub-Kahan bidiagonalization, applied to (\ref{eq:linsys}) or to (\ref{tikhx0}) in a hybrid fashion. Section \ref{sec:inexact} introduces the new inexact solvers based on the inexact Golub-Kahan decomposition, and develops some theory (including bounds for the amount of inexactness) that relates them to their exact counterparts. 
Section \ref{sec:blind} describes how the new inexact methods can be applied to the blind deblurring problem, including computable strategies to define and bound inexactness; the results of some numerical tests are also displayed. Finally, Section \ref{sec:end} outlines some concluding remarks.

\emph{Notations}. In the following, matrices are denoted by uppercase letters and the $(i,j)$th entry of a matrix $G$ is denoted by $[G]_{i,j}$. The letter $I$ denotes the identity matrix, whose size should be clear from the context; $e_j$ is the $j$th canonical basis vector (i.e., the $j$th column of $I$). $\mcR(G)$ denotes the range (column space) of the matrix $G$. Vectors are denoted by lowercase letters and the $i$th entry of a vector $g$ is denoted by $[g]_i$. The letter $r$ is routinely used to denote residual vectors.

%

\section{Combining Krylov methods and Tikhonov regularization}\label{sec:exact}

In this section we assume that there is no inexactness. 
We first recall the Golub-Kahan (GK) algorithm and its relations with the symmetric Lanczos algorithm, and we then explain how both algorithms can be employed to approximate the solution of problems (\ref{eq:linsys}) and (\ref{tikhx0}). 

Starting from
\begin{equation} \label{eq:Lanczos0}
u_1 = r_0/\|r_0\|=r_0/\beta,\qquad \|v\|v=A^Tu_1\,,\; v_1 = v/\|v\|=v/\alpha_1\,,
\end{equation}
the $i$th GK iteration 
computes
\begin{equation}\label{LaczosVect}
\beta_{i+1}u_{i+1} =Av_i - \alpha_iu_i\,,\qquad \alpha_{i+1} v_{i+1} = A^Tu_{i+1} -\beta_{i+1} v_{i},  \quad i =1,2,\dots,n-1,
\end{equation}
with $\alpha_{i+1}$ and $\beta_{i+1}$ chosen so that $\|v_{i+1}\|=1$ and $\|u_{i+1}\|=1$, respectively. In the following we assume that $\alpha_{i+1}\neq 0$ or $\beta_{i+1}\neq 0$, $i=1,\dots,n-1$  (i.e., GK does not break down), so that relations (\ref{LaczosVect}) are always well-defined; we also note that, in practice, GK applied to (\ref{eq:linsys}) is feasible only if $i\ll n$ iterations are computed. After $k$ GK iterations are performed, one can write partial matrix factorizations of the form
\begin{eqnarray} \label{eq:Lanczos1}
A V_k 
=U_{k+1} \bB_k \,, 
\qquad A^T U_{k+1} = V_{k+1} B_{k+1}^T\,,
\end{eqnarray} 
where $V_{k+1}=[v_1,\dots,v_k,v_{k+1}] \in \mathbb{R}^{n \times (k+1)}$ and $U_{k+1}=
[u_1,\dots,u_{k+1}] \in \mathbb{R}^{m \times (k+1)}$, 
are matrices whose orthonormal columns span the Krylov subspaces \linebreak[4]$\mathcal{K}_{k+1}(A^{T} A,A^{T} r_0)$ and $\mathcal{K}_{k+1}(A A^{T}, r_0)$, respectively; 
$B_{k+1}\in \mathbb{R}^{(k+1) \times (k+1)}$ is the lower bidiagonal matrix having 
$[B_k]_{i,i}=\alpha_i$, $i=1,\dots,k+1$ and $[B_k]_{i+1,i}=\beta_{i+1}$, \linebreak[4]$i=1,\dots,k$; $\bB_k\in \mathbb{R}^{(k+1) \times k}$ is obtained by removing the last column of $B_{k+1}$. 

The symmetric Lanczos \cite[Chapter 6]{Saad2003} and the GK algorithms are closely related: indeed,  
multiplying the first expression in (\ref{eq:Lanczos1}) from the left by $A^T$, and using again the second equation in (\ref{eq:Lanczos1}), one obtains
\begin{equation}
\underbrace{A^T A}_{=:\whA} V_k = A^T U_{k+1} \bar{B}_{k} = V_{k+1} \underbrace{B_{k+1}^T \bar{B}_{k}}_{=:\bar{T}_k} = V_{k} \underbrace{\bar{B}_{k}^T \bar{B}_{k}}_{=:{T}_k}+\alpha_{k+1} \beta_{k+1} v_{k+1}e^T_{k},\label{SymLanczos2}
\vspace{-0.2cm}
\end{equation}
so that $V_k$ can be regarded as the matrix generated by performing $k$ steps of the symmetric Lanczos algorithm applied to $A^T A$, with initial vector $A^T r_0$. We remark that the symmetric Lanczos algorithm is mathematically equivalent to the Arnoldi algorithm, the symmetric tridiagonal matrix ${T}_k$ being linked to the fact that $A^TA$ is symmetric. 

LSQR is an iterative solver for (\ref{eq:linsys}) based on GK bidiagonalization (\ref{eq:Lanczos1}). The $k$th iteration of LSQR computes
\begin{equation}\label{eq:LSQR1}
s_k = \arg\min_{s\in \R^k}\|\bB_k s - \beta e_1\|\,,\quad\mbox{i.e.,}\quad (\bB_k^T\bB_k)s_k = \bB_k^T(\beta e_1)\,,
\end{equation}
and takes $x_k=x_0+z_k=x_0+V_ks_k$. Equivalently, using (\ref{eq:Lanczos0}) and (\ref{eq:Lanczos0}), LSQR computes 
\begin{equation}\label{eq:LSQR2}
q_k=\arg\min_{q\in \mcR(U_{k+1}\bB_k)=\mcR(AV_{k})}\|q - r_0\|\,.
\end{equation}
In other words, the $k$th iteration of LSQR projects the solution of (\ref{eq:linsys}) onto $\mcR(V_k)$ orthogonally to $\mcR(AV_k)$. LSQR is also mathematically equivalent to CGLS, i.e., CG 
method applied to the normal equations associated to (\ref{eq:linsys}). Namely, 
(\ref{eq:LSQR1}) can be also expressed as\vspace{-0.4cm}
\begin{equation}\label{eq:cg}
\overbrace{V_k^T(A^TA)V_k}^{=\bB_k^T\bB_k} s_k\quad
=\overbrace{V_k^TA^Tr_0}^{\bB_k^TU_{k+1}^Tr_0=\bB_k^T\|r_0\|e_1}\,,
\end{equation}
where equation (\ref{SymLanczos2}) has been used to reformulate the leftmost quantity, and relations (\ref{eq:Lanczos0}) and (\ref{eq:Lanczos1}) have been used to get the rightmost equalities. This essentially means that CGLS computes \vspace{-0.4cm}
\begin{equation}\label{eq:cg2}
q_k\in\mcR(A^TAV_k)=\mcR(V_{k+1}\bar{T}_k)\quad\mbox{such that}\quad \overbrace{A^Tr_0}^{=:\whr_0} - q_k \perp \mcR(V_k)\,,
\end{equation}
i.e., CGLS projects the solution of (\ref{eq:linsys}) onto $\mcR(V_k)$ orthogonally to $\mcR(V_k)$. 

LSQR and CGLS can be adopted in a hybrid fashion to solve (\ref{tikhx0}), too. 
In this setting, assuming for now that $\lambda$ is fixed, 
we are faced with many equivalences 
that involve the interplay of regularization and projection, and leverage properties of Krylov basis vectors (such as orthonormality or shift invariance). 
Despite the specific formulation, the $k$th iteration of all the methods computes 
\begin{equation}\label{xapprox}
x_{\true}\approx x_{\lambda,k}=x_0+V_ks_{\lambda,k}\in\Kry_k(A^TA, A^Tr_0)=\Kry_k(A^TA+\lambda^2I, A^Tr_0)\,.
\end{equation}
%
Starting from (\ref{tikhx0}), and exploiting (\ref{eq:Lanczos0}) and (\ref{eq:Lanczos1}), 
we can write
\begin{equation}
s_{\lambda,k}= 
\arg\min_{s\in\R^k}
\left\|
AV_k s - r_0\right\|^2 + 
\lambda^2  \left\| V_ks\right\|^2
=
\arg\min_{s\in\R^k}
\left\|
\bB_k s - \beta e_1\right\|^2 + 
\lambda^2  \left\| s\right\|^2\,.\label{eq:projregpb}
\end{equation}
Alternatively, starting from the reformulation of (\ref{tikhx0}) as an \emph{augmented LS problem}, and exploiting similar properties, 
we can write
\begin{equation}\label{eq:projregaugpb}
s_{\lambda,k} = 
\arg\min_{s\in\R^k}\:
\left\|\left[\begin{array}{c}
A \\
\lambda I
\end{array}
\right]V_k s-
\left[\begin{array}{c}
r_0 \\
0
\end{array}
\right]
\right\|^2 
= \arg\min_{s\in\R^k}\:
\left\|
\left[
\begin{array}{c}
\bB_k \\
\lambda I
\end{array}
\right] s -
\left[\begin{array}{c}
\beta e_1 \\
0
\end{array}
\right]
\right\|^2\,.
\end{equation}
Obviously (\ref{eq:projregpb}) and (\ref{eq:projregaugpb}) are equivalent and, in particular, the leftmost quantities in both equations can be regarded as Tikhonov-regularized versions 
of the projected LS problem (\ref{eq:LSQR1}) solved by LSQR. Therefore, first regularizing (considering a penalized or augmented LS problem) and then projecting is equivalent to first projecting and then regularizing (considering a penalized or augmented LS problem, respectively). Note that, since $V_k$ is generated with respect to $A$ and $r_0$ (i.e., it is independent of $\lambda$), one can potentially change $\lambda$ at each GK iteration so that, if $\lambda=\lambda_k$ at the $k$th iteration 
of the hybrid methods (\ref{eq:projregpb}) and (\ref{eq:projregaugpb}), 
an instance of problems (\ref{tikhx0}) 
with $\lambda=\lambda_k$ is approximated. This proves strategic in case a suitable value of $\lambda$ is not known a priori (see \cite{chungframe}), and to adapt the amount of regularization to the $k$th projected problem. 
%
Finally, we can consider the normal equations formulation associated to the Tikhonov augmented LS problem 
and, similarly to (\ref{eq:cg}), we compute
\begin{eqnarray}
s_{\lambda,k} &=& \left(V_k^T(A^TA + \lambda^2 I)V_k\right)^{-1}V_k^T(A^Tr_0) \label{eq:CGderivation}\\
&=& ({T}_k + \lambda^2 I)^{-1}\|A^Tr_0\|e_1= \left(\bB_k^T\bB_k + \lambda^2 I\right)^{-1}\bB_k^T(\beta e_1)\,,\nonumber
\end{eqnarray}
where we have exploited (\ref{eq:Lanczos0}) and (\ref{SymLanczos2}). 
Note that $s\lamk$ in (\ref{eq:CGderivation}) expresses the solution to the normal equations associated to (\ref{eq:projregaugpb}), so that (\ref{eq:CGderivation}) is equivalent to (\ref{eq:projregpb}) and (\ref{eq:projregaugpb}). Also in this case, by exploiting the shift-invariance of Krylov subspaces (\ref{xapprox}), one can see that applying CG to the shifted (regularized) normal equations is equivalent to shifting (regularizing) the projected normal equations, 
implying that $\lambda$ can be adaptively set during the iterations. 
\section{Combining inexact Krylov methods and Tikhonov regularization}\label{sec:inexact}
In this section we first present an inexact Golub-Kahan decomposition, and we show how it can be employed to solve linear systems of the form (\ref{eq:linsys}). 
We then derive a couple of strategies to combine the inexact Golub-Kahan decomposition and Tikhonov regularization. Unfortunately, only a few of the equivalences presented in Section \ref{sec:exact} for the exact case extend to the inexact case. 

\subsection{Inexact Golub-Kahan (iGK) decomposition}\label{Section:iGKB}
Assume that the actions of $A$ and $A^T$ are just approximately available. Then the solution of linear systems of the form (\ref{eq:linsys}) can be efficiently approximated using methods based on the inexact Golub-Kahan (iGK) algorithm that, 
at the $i$th iteration, 
only uses the available actions of $(A+E_i)$ and $(A+F_i)^T$. 
Starting from
\begin{equation}\label{eq:inex_gkb_0}
u_1 = r_0/\|r_0\|=r_0/\beta,\quad \|v\|v=(A+F_1)^Tu_1\,,\quad v_1 = v/\|v\|=v/[L]_{1,1}\,,
\end{equation}
the $i$th iteration of the inexact Golub-Kahan (iGK) algorithm computes
\begin{equation}\label{eq:iGKDvect}
\begin{array}{lll}
u = (A+E_{i})v_i, & u = (I - U_iU_i^T)u, & u_{i+1}=u/\|u\|\,,\\
v = (A+F_{i+1})^Tu_{i+1}, & v = (I - V_iV_i^T)v, & v_{i+1}=v/\|v\|\,,
\end{array}
\end{equation}
where $U_i=[u_1,\dots,u_i]\in\R^{m\times i}$, $V_i=[v_1,\dots,v_i]\in\R^{n\times i}$ are matrices with orthonormal columns. Note that, if $x_0\neq 0$, the starting vector $r_0$ in (\ref{eq:inex_gkb_0}) may already be affected by some inexactness, as $r_0=b-(A+E_0)x_0$; by committing a slight abuse of notation, here and in the following we will still denote such quantity by $r_0$, even if it may be different from the one appearing in (\ref{tikhx0}), (\ref{eqtikhy}), and (\ref{eq:Lanczos0}). 
After $k$ iGK iterations are performed, one can write partial matrix factorizations of the form 
\begin{equation}\label{eq:inex_gkb_1}
\begin{array}{rcl}
\left[ (A+E_1)v_1,..., (A+E_{k})v_k \right] & = & U_{k+1} M_k \,, \\
\left[ (A+F_1)^T u_1,..., (A+F_{k+1})^T u_{k+1} \right] & = & V_{k+1} L_{k+1}^T\,,
\end{array}
%
\end{equation}
where 
$M_k\in\R^{(k+1)\times k}$ is upper Hessenberg with $[M_k]_{j,i}=u_j^Tu$ and $[M_k]_{i+1,i}=\|u\|$ for $j\leq i\leq k$, and $L_{k+1} \in \R^{(k+1) \times (k+1)}$ is lower triangular with $[L_{k+1}]_{i+1,j}=v_j^Tv$ and $[L_{k+1}]_{i+1,i+1}=\|v\|$ for  $j\leq i\leq k$. Note that, to impose orthogonality, $M_k$ and $L_{k+1}$ have to be considered rather than the simpler bidiagonal $B_k$ (and its variant) as in (\ref{eq:Lanczos1}). 
The above relations can be re-written in the following even more compact form
\begin{equation}\label{eq:inex_gkb_2}
\begin{array}{lcl}
(A+\mathcal{E}_k)V_k &=& U_{k+1} M_k \\
(A+{\mathcal{F}}_{k+1})^T U_{k+1} & = & V_{k+1} L_{k+1}^T
\end{array},
\quad\mbox{where}\quad
\begin{array}{lcl}
\mathcal{E}_k &=&\sum_{i=1}^k E_i v_i v_i^T\\
\mathcal{F}_{k+1}&=&\sum_{i=1}^{k+1} F_i u_iu_i^T 
\end{array}.
\end{equation}
Such partial decompositions involving the matrices $A$ and $A^T$ were first introduced in \cite{inexactGKB} in the framework of matrix function computations; also, the iGK 
decomposition bears similarities to the flexible Golub-Kahan decomposition introduced in \cite{JulianneSilvia}.

Similarly to relation (\ref{SymLanczos2}) in the exact case, one can use (\ref{eq:inex_gkb_2}) to devise an inexact relation involving the matrix $A^TA$ by considering 
\begin{eqnarray}
(\overbrace{A^TA}^{=\widehat{A}} +\overbrace{{\mathcal{F}}_{k+1}^TA + A^T \mathcal{E}_{k} + {\mathcal{F}}_{k+1}^T\mathcal{E}_{k}}^{=:\widehat{\mathcal{E}}_{k}})V_k &=& (A+{\mathcal{F}}_{k+1})^T (A+\mathcal{E}_{k})V_k \nonumber\\
&=&(A+{\mathcal{F}}_{k+1})^TU_{k+1}M_k
=V_{k+1}\overbrace{L_{k+1}^TM_k}^{=:\widehat{H}_k} \label{eq:iLanczos}\\
&=&V_k \bar{L}_k^T M_k + [\widehat{H}_k]_{k+1,k}v_{k+1} e_k^T,\nonumber
\end{eqnarray}
where $\bar{L}_k\in\R^{(k+1)\times k}$ is the matrix obtained by removing the last column of $L_{k+1}$. With respect to the exact case (\ref{SymLanczos2}), we note that, in the fourth and fifth of the above equalities, an upper Hessenberg matrix $\whH_k\in\R^{(k+1)\times k}$ appears (instead of a symmetric tridiagonal matrix $T_k$). This is a consequence of the fact that $(\widehat{A}+{\widehat{\mathcal{E}}_{k}})$ is (in general) non symmetric (unless the matrices $A$, $\mathcal{E}_{k}$, $\mathcal{F}_{k+1}$ all commute). 
Moreover, while (\ref{SymLanczos2}) was a particular case of the 
Arnoldi decomposition (applied to $\whA$), 
the inexact Lanczos (iLanczos) decomposition (\ref{eq:iLanczos}) cannot generally  
be regarded as a particular case of an inexact Arnoldi (iArnoldi) decomposition (associated to $\whA$). 
Indeed, starting from $v_1^{iA}=v_1$ as in (\ref{eq:inex_gkb_0}), the $i$th iArnoldi iteration computes
\[
v = (A+F_{i+1})^T(A+E_{i})v_i^{iA}, \quad v = (I - V_i^{iA}(V_i^{iA})^T)v, \quad v_{i+1}^{iA}=v/\|v\|\,,
\]
where $V_i^{iA}=[v_1^{iA},\dots,v_i^{iA}]\in\R^{n\times i}$ has orthonormal columns. After $k$ iArnoldi iterations are performed, one can write a partial matrix factorization of the form 
\begin{equation}\label{eq:inex_arnoldi}
\left[ (A+F_2)^T(A+E_1)v_1^{iA},..., (A+F_{k+1})^T(A+E_{k})v_k^{iA} \right] =  V_{k+1}^{iA} H_k^{iA}\,, 
\end{equation}
where $H_k^{iA}\in\R^{(k+1)\times k}$ is upper Hessenberg with $[H_k^{iA}]_{j,i}=v_j^Tv$ and $[H_k^{iA}]_{i+1,i}=\|v\|$ for $j\leq i\leq k$. More compactly,
\begin{equation}\label{eq:inex_arnoldi_2}
(\whA+\widehat{\mathcal{E}}_k^{iA})V_k^{iA}= V_{k+1}^{iA} H_k^{iA},
\;\mbox{where}\;
\widehat{\mathcal{E}}_k^{iA} =\sum_{i=1}^k (A^TE_i + F_{i+1}^TA + F_{i+1}^TE_i)v_i^{iA} (v_i^{iA})^T.
\end{equation}
By comparing (\ref{eq:iLanczos}) and (\ref{eq:inex_arnoldi_2}), 
one can see that, for $j=1,\dots,k$, iLanczos computes
\begin{equation}\label{1stepiLanczos}
A^TAv_j + \underbrace{A^TE_jv_j + \mcF_{k+1}U_{k+1}M_ke_j}_{=A^TE_jv_j+\sum_{i=1}^{j+1}[M_k]_{i,j}F_i^Tu_i=:\,\widehat{e}_j} 
= \sum_{i=1}^{j+1}[\whH_k]_{j,i}v_i\,,
\end{equation}
while iArnoldi computes
\[
A^TAv_j^{iA} + A^TE_jv_j^{iA} + F_{j+1}^TAv_j^{iA} + F_{j+1}^TE_jv_j^{iA}=\sum_{i=1}^{j+1}[H_k^{iA}]_{j,i}v_i^{iA}\,.
\]
It is clear that the two expressions above are generally different. They however coincide in specific instances, e.g., when $F_{j+1}=0$, $j=1,\dots,k$, i.e., when the matrix-vector products with $A^T$ are computed exactly. 
%

\subsection{Linear solvers based on the iGK decomposition} 
We define the inexact LSQR (iLSQR) method to be an iterative solver for (\ref{eq:linsys}) that, at the
$k$th iteration, computes
\begin{equation}\label{eq:iLSQR2}
q_k=\arg\min_{q\in \mcR(U_{k+1}M_k)}\|q - r_0\|\,.
\end{equation}
Thanks to the orthonormality of the columns of $U_{k+1}$ and (\ref{eq:inex_gkb_2}), the equalities
\[
\|M_k s - \beta e_1\|=\|U_{k+1}M_k s - r_0\|=\|(A+\mcE_k)V_k s - r_0\|
\]
hold, so that iLSQR equivalently computes 
\begin{equation}\label{eq:iLSQR1}
s_k = \arg\min_{s\in \R^k}\|M_k s - \beta e_1\|\,,\quad\mbox{i.e.,}\quad (M_k^TM_k)s_k = M_k^T(\beta e_1)\,,
\end{equation}
and takes $x_k=x_0+z_k=x_0+V_ks_k$. Although (\ref{eq:iLSQR1}) and (\ref{eq:iLSQR2}) are formally equivalent to (\ref{eq:LSQR1}) and (\ref{eq:LSQR2}), respectively (i.e., the analogous relations written for LSQR), because of the presence of $\mathcal{E}_k$ in the first equation in (\ref{eq:inex_gkb_2}), the $k$th iLSQR iteration does not minimize the exact residual $\|r_0-Az\|$ among the vectors $z\in\mcR(V_k)$, and $\mcR(V_k)$ is not a Krylov subspace anymore. Note that this is analogous to what happens in the case of GMRES and inexact GMRES; see \cite{inexact1} for more details. 

We define the inexact CGLS (iCGLS) method to be an iterative solver for (\ref{eq:linsys}) that, at the $k$th iteration, computes
\begin{equation}\label{eq:icg0}
q_k\in\mcR(V_{k+1}\whH_{k})\quad\mbox{such that}\quad (A+{\mathcal{F}}_{k+1})^Tr_0 - q_k 
\perp \mcR(V_k)\,.
\end{equation}
Note that, contrarily to exact case (\ref{eq:cg2}), iCGLS does not impose an orthogonality condition on the exact normal equation residual $A^Tr_k$, and additional care should be taken because $(A+{\mathcal{F}}_{k+1})^Tr_0=(A+F_1)^Tr_0$ is also potentially affected by some errors. 
Equivalently, instead of imposing (\ref{eq:icg0}), one may impose 
\begin{equation}\label{eq:icg1}
V_k^T(\whA+\whmcE_k)V_k s_k = V_k^T(A+{\mathcal{F}}_{k+1})^Tr_0\,,
\end{equation}
directly, so that iCGLS solves
\begin{equation}\label{eq:icg}
\bar{L}_k^T M_k s_k = [\bar{L}_k]_{1,1}\beta e_1
\end{equation}
and takes $x_k=x_0+z_k=x_0+V_ks_k$. 

We must stress that 
iLSQR is not equivalent to iCGLS anymore 
(this is evident comparing equations (\ref{eq:iLSQR1}) and (\ref{eq:icg})); also, differently to the exact case, $M^T_k \beta e_1 \neq \|A^T r_0\| e_1$. However, the $k$th iteration of both iLSQR and iCGLS computes
\begin{equation}\label{ixapprox}
x_{k}=x_0+V_ks_{k}\,,\quad\mbox{where $V_k$ is defined in (\ref{eq:inex_gkb_2}) or (\ref{eq:iLanczos}).}
\end{equation} 

In the following we will introduce two different strategies to combine Tikhonov regularization and inexact Krylov methods, based on iLSQR and iCGLS, respectively. 
\subsection{A hybrid method based on iLSQR}\label{sec:hybridiLSQR}
A first way of combining iGK and Tikhonov regularization consists in applying the iLSQR condition (\ref{eq:iLSQR2}) 
to the 
augmented LS problem formulation of Tikhonov regularization (\ref{tikhx0}) assuming, for the moment, that $\lambda$ is fixed. Namely, 
we extend relations (\ref{eq:projregaugpb}) to the inexact case by computing, at the $k$th 
iteration,
\begin{eqnarray}\label{eq:iprojregaugpb}
q_{\lambda,k}=\arg\min_{q \in \mathcal{R}(W_{\lambda,k})} \left\|
q-
\left[\begin{array}{c}
r_0 \\
0
\end{array}
\right]
\right\|, \quad \text{where} \quad W_{\lambda,k} =\left[\begin{array}{c}
U_{k+1} M_k \\
\lambda V_k
\end{array} \right].
\end{eqnarray}
Exploiting the relation
\begin{equation}\label{iGKBaug}
\left[\begin{array}{c}
A + \mcE_k\\
\lambda I
\end{array}
\right]V_k = 
\left[\begin{array}{cc}
U_{k+1} & 0 \\
0 & V_k
\end{array}
\right]
\left[\begin{array}{c}
M_{k} \\
\lambda I
\end{array}
\right]=
\left[\begin{array}{c}
U_{k+1}M_{k} \\
\lambda V_k
\end{array}
\right]\,,
\end{equation}
which is a trivial extension of (\ref{eq:inex_gkb_2}), and recalling the definition of $u_1$ in (\ref{eq:inex_gkb_0}), it is easy to see that solving problem (\ref{eq:iprojregaugpb}) is equivalent to computing $q_{\lambda,k}=W_{\lambda,k}s_{\lambda,k}$, where
\begin{eqnarray}
s_{\lambda,k} &=& 
\arg \min_{s\in\R^k} \left\|
\left[\begin{array}{c}
M_k \\
\lambda I
\end{array} \right] s -
\left[\begin{array}{c}
\beta e_1 \\
0
\end{array}
\right] 
\right\|^2
\!\!\!\!= \arg \min_{s\in\R^k} \|M_ks - \beta e_1\|^2 + \lambda^2 \|s\|^2 \label{eq:iprojregaugpb_1}\\ 
&=& (M^T_k M_k+ \lambda^2 I)^{-1} M^T_k (\beta e_1).\label{eq:iprojtikhne}
\end{eqnarray}
In the following we refer to method in (\ref{eq:iprojregaugpb}) or (\ref{eq:iprojregaugpb_1})-(\ref{eq:iprojtikhne}) as hybrid-iLSQR. Looking at the above equations, it is evident that first regularizing and then projecting (i.e., the path that we just followed) is equivalent to first projecting and then regularizing. Indeed, the same problem (\ref{eq:iprojregaugpb_1}) may be 
obtained by first applying iLSQR to (\ref{eq:linsys}) and then regularizing the projected LS problem (\ref{eq:iLSQR1}) (using the augmented LS formulation of Tikhonov regularization). Note that, taking inexactness into account and using relations (\ref{eq:inex_gkb_2}) and (\ref{iGKBaug}), one can link (\ref{eq:iprojregaugpb_1}) to the full-dimensional problem 
\begin{eqnarray}
\label{opthiLSQR}
s\lamk = 
\arg \min_{s \in \mathbb{R}^k} \left\|
\left[\begin{array}{c}
A + \mcE_k\\
\lambda I
\end{array}
\right]V_ks -  
\left[\begin{array}{c}
r_0 \\
0
\end{array}
\right]\right\|^2 .
\end{eqnarray}
Looking at the above formulation it is evident that the optimality properties of hybrid LSQR (\ref{eq:projregpb}) and hybrid-iLSQR are different, as the functional minimized by the latter is an error-corrupted version of the augmented-LS Tikhonov functional. 
\subsection{A hybrid method based on iCGLS}\label{sec:hybridiCGLS}
A second way of combining iGK and Tikhonov regularization consists in 
extending condition (\ref{eq:icg0}) to the normal equations associated to the 
augmented LS problem formulation of Tikhonov regularization (\ref{tikhx0}) assuming, for the moment, that $\lambda$ is fixed. Namely, at the $k$th iteration, we impose 
\begin{equation}\label{eq:iCGderivation1}
q\lamk \in \mathcal{R}(W\lamk)=\mcR(V_{k+1}(\widehat{H}_{k} + \lambda^2\bar{I}))\,,\quad
(A+\mcF_{k+1})^Tr_0 - q\lamk  \perp \mathcal{R}(V_k)\,,
\end{equation}
where $\bar{I}\in\R^{(k+1)\times k}$ denotes the identity matrix of order $(k+1)$ without its last column. Equivalently, we can apply 
the iCGLS condition (\ref{eq:icg1}) 
to the normal equations associated to the 
augmented LS problem formulation of Tikhonov regularization (\ref{tikhx0}) assuming, for the moment, that $\lambda$ is fixed. That is, 
we compute
\begin{eqnarray}
s_{\lambda,k} &=& \left(V_k^T(\whA + \whmcE_k + \lambda^2 I)V_k\right)^{-1}V_k^T(A+\mcF_{k+1})^Tr_0 \label{eq:iCGderivation}\\
&=& (\bL_k^TM_k + \lambda^2 I)^{-1}[\bL_k]_{1,1}\beta e_1\,,\nonumber
\end{eqnarray}
and then take $x\lamk=x_0+V_ks\lamk$. In the following we refer to the method in (\ref{eq:iCGderivation1}) or (\ref{eq:iCGderivation}) as hybrid-iCGLS method. 
Looking at equation (\ref{eq:iCGderivation}), it is evident that applying iCGLS to the shifted normal equations (i.e., the path that we just followed) is equivalent to shifting the projected normal equations (so that, similarly to hybrid LSQR, regularization and projection are interchangeable). However note that, differently from hybrid LSQR (\ref{eq:CGderivation}), hybrid-iCGLS projects the error-corrupted normal equations. 
%

\subsection{Remarks about iLSQR, iCGLS, and their hybrid counterparts} \label{sec: theory}
As already mentioned in the previous sections, one of the upsides of using hybrid methods consists in the fact that they allow efficient and adaptive regularization parameter choice on the fly. This is essentially linked to the shift-invariance property of the approximation subspace for the solution, which is the same for LSQR-based and CGLS-based methods; see (\ref{xapprox}). Before 
considering possible parameter choice strategies for hybrid methods based on iGK, it is therefore natural to assess if shift invariance is still enjoyed by the inexact approximation subspace for the solution; see (\ref{ixapprox}). 
%
In the iCGLS case we can state the following.
\begin{proposition}\label{prop:spacei}
Assume that $k$ iLanczos iterations (\ref{eq:iLanczos}) have been performed, so that the partial decomposition (\ref{eq:iLanczos}) can be written. Assume that $\{\lambda_i\}_{i=1,\dots,k}\subset \R^+_0$ is such that $\lambda_i\neq\lambda_j$ if $i\neq j$. If $\whmcE_k$ is independent of $\lambda_i$ (or, equivalently, both $\mcE_k$ and $\mcF_{k+1}$ are independent of $\lambda_i$), $i=1,\dots,k$, then 
the subspace $\mcR(V_k)$ is shift-invariant. Moreover, if $\mcF_{k+1}=0$ and 
\begin{equation}\label{eq:shiftinvdef}
\whmcE_{i,\tilde{\Lambda}_k} =\sum_{j=1}^i(E_j+ \tilde{\lambda}_j^2I)v_jv_j^T=
\mcE_i + V_i\tilde{\Lambda}_i^2 V_i^T,\quad
\tilde{\Lambda}_i=\mbox{\emph{diag}}(\tilde{\lambda}_1,\dots,\tilde{\lambda}_i)\,,
\end{equation}
then 
the subspace $\mcR(V_k)$ is shift-invariant.
\end{proposition}
\begin{proof}
The proof follows directly from (\ref{eq:iLanczos}). 
Indeed, for the first statement, 
\[
(\whA + \whmcE_i)V_i + \lambda_i^2V_i = V_{i+1}\whH_i+ \lambda_i^2V_{i+1}\bar{I}=V_{i+1}(\whH_i+ \lambda_i^2\bar{I})\,,\quad i=1,\dots,k\,.
\]
For the second statement,  
\begin{eqnarray*}
(\whA + \whmcE_{i,\tilde{\Lambda}_k} + \lambda_i^2I)V_i &=& (\whA + \whmcE_{i})V_i+ V_i\tilde{\Lambda}^2_iV_i^TV_i + \lambda_i^2IV_i
= V_{i+1}\whH_i + V_i\tilde{\Lambda}_i^2 + \lambda_i^2V_i\\
&=&
V_{i+1}\left(\whH_i + \left[\begin{array}{c}
\tilde{\Lambda}_i^2\\
0
\end{array}
\right]+\lambda_i^2\bar{I}\right)\,,\quad i=1,\dots,k\,.
\end{eqnarray*}
%
\end{proof}

We now derive bounds on norms of relevant quantities computed by the exact and inexact solvers. 
We start by studying the relationship between the exact residual $r_k^e$ and the inexact residual $r_k$, $k=1,2,\dots$, 
extending to iLSQR and iCGLS 
the relations derived in \cite{inexact1} for inexact FOM and GMRES. 
When considering iLSQR (\ref{eq:iLSQR1}), it is natural to bound the distance between the exact and the inexact residuals as follows 
\begin{eqnarray}
\|r_k^e-r_k\| = \|r_0^e-Ax_k - (r_0 - (A+\mcE_k)x_k)\|&=&\|E_0x_0 + \mcE_kV_ks_k\|\nonumber\\
&\leq& \|E_0x_0\|+\sum_{l=1}^k\|E_l\|\,|[s_k]_l|\label{iLSQRdiffbound}
\end{eqnarray}
%
Directly from the relation above, the following bound for the norm of exact residual can be derived
\begin{eqnarray}
\|r_k^e\| &\leq& \|r_k\| + \|E_0x_0\|+\sum_{l=1}^k\|E_l\|\,|[s_k]_l|\,.\label{iLSQRbound}
\end{eqnarray}
Note that the residuals $r_k$, $k=1,2\dots$ are the only ones that we can assume available (and whose norms can be efficiently computed, as $\|r_k\|=\|M_ks_k - \beta e_1\|$). In the following we focus on bounds for the norm of the exact residuals only. 

Similar but somewhat more complicated estimates hold when considering iCGLS, as both errors in $A$ and $A^T$ must be included; in particular, the right-hand-side vector in equation (\ref{eq:icg1}) is affected by errors both in $r_0$ and in $A^T$. Indeed, 
\[
\whr_0=(A+F_{1})^Tr_0=\underbrace{A^T(b - Ax_0)}_{=:\whr_0^e} + \underbrace{F_1^Tb}_{=:\whb_0} - \underbrace{(F_1^TA + A^TE_0 + F_1^TE_0)x_0}_{=:\whE_0x_0}\,.
\]
When bounding the normal equations residual norm, one should then consider
\begin{eqnarray}
\|\underbrace{\whr_0^e - \whA V_ks_k}_{=:\whr_k^e}\| &=& \|\whr_0 -\whb_0 + \whE_0x_0 - (\whA+\whmcE_k)V_ks_k + \whmcE_kV_ks_k\|\nonumber\\
&\leq& \| [L_k]_{1,1}\beta e_1 - \whH_ks_k\| + \|\whb_0\| + \|\whE_0x_0\|\label{iCGLSbound}\\
& &  
+ \sum_{j=1}^k\|A^TE_j\||[s_k]_j| + \sum_{j=1}^k\left(\sum_{i=1}^{j+1}\left| [M_k]_{i,j}\right| \|F_i^T\|\right)|[s_k]_j|\,,\nonumber
\end{eqnarray}
where the last two sums are obtained applying standard bounds to $\widehat{e}_j$ in (\ref{1stepiLanczos}). 

Contrarily to the well-posed case, when solving ill-posed problems using iLSQR and iCGLS, one should not expect $r_k^e$ and $\whr_k^e$ to be close to zero, as this would imply data overfitting (recall the discussion in Section \ref{sec:exact}). As a consequence, one can allow more  inexactness. 
The above relations extend to iLSQR and iCGLS used in a hybrid fashion. 

When considering hybrid-iLSQR (\ref{eq:iprojregaugpb_1}) with a fixed $\lambda$, let us define
\[
r\lamk^e=:\left[\begin{array}{c}
r_0^e \\
0
\end{array}\right]-\left[\begin{array}{c}
A \\
\lambda I
\end{array}\right]V_k s\lamk,
\quad
r\lamk=:\left[\begin{array}{c}
r_0^e \\
0
\end{array}\right]-\left[\begin{array}{c}
A + \mcE_k\\
\lambda I
\end{array}\right]V_k s\lamk\,,
\]
so that
\begin{eqnarray}
\left\Vert r\lamk^e \right\Vert^2&=& \left\|
r\lamk
+
\left[\begin{array}{c}
\mcE_k\nonumber\\
0
\end{array}\right]V_k s\lamk
+
\left[\begin{array}{c}
E_0\\
0
\end{array}\right]x_0
\right\|^2\\
&\leq& \|r\lamk\|^2 + \|E_0x_0\|^2 + \sum_{l=1}^k\|E_l\|^2|[s\lamk]_l|^2\,.\label{hiLSQRbound}
\end{eqnarray}
In order for the exact and inexact residual norms to be close, 
one should estimate the desired magnitude of $\|r\lamk^e\|^2$. Contrarily to $\|r_{k}^e\|$, even running $n$ iterations of the (exact) hybrid LSQR would not deliver a value of $\|r_{\lamk}^e\|$ close to zero. 
We have that, ideally, if $x_\true-x_0=:z_\true$ and $e$ were known, the most desirable value of the objective function in (\ref{tikhx0}) 
would be 
\[
\|Az_\true -r_0^e\|^2 +\lambda^2\|z_\true\|^2 =  \|e\|^2 +\lambda^2\|z_\true\|^2\,.
\]
Now, depending on the application, estimates for $\|e\|^2$ and $\|z_\true\|^2$ may be available. If not, one may use a noise estimator for the former (see, e.g., \cite{noiseest}) and, recalling that
\[
\|r_0^e\|^2\leq \|A\|^2\|z_\true\|^2 + \|e\|^2\quad \mbox{(directly from (\ref{eq:linsys}))},\quad\mbox{and}\quad\|r_0\|\leq \|r_0^e\|+\|E_0x_0\|\,,
\]
take the lower bound
\begin{eqnarray}\label{zbound}
\|z_\true\|^2 &\geq& \frac{1}{\|A\|^2}(\|r_0\|^2 - \|e\|^2 -\|E_0x_0\|^2) 
%
\end{eqnarray}
for the latter, where $\|A\|$ should be estimated as well. 


When considering hybrid-iCGLS (\ref{eq:iCGderivation}) with a fixed $\lambda$, the estimates are very similar to the ones written above in the iCGLS case. 
Namely, 
\begin{eqnarray}
\|\whr\lamk^e\| 
&\leq&\|\whr\lamk\| + \|\whb_0\| + \|\whE_0x_0\| \label{hiCGLSbound}\\
& & + \sum_{j=1}^k\|A^TE_j\||[s_{\lamk}]_j| + \sum_{j=1}^k\left(\sum_{i=1}^{j+1}\left| [M_k]_{i,j}\right| \|F_i^T\|\right)|[s_{\lamk}]_j|\nonumber
\end{eqnarray}
where $\whr\lamk^e:=\whr_0^e - (\whA + \lambda^2 I)x\lamk$ and $\whr\lamk=\whr_0  - ((\whA + \whmcE_k) + \lambda^2 I)x\lamk$. 
We should however stress that, differently from all the other estimates so far derived, $\|\whr\lamk^e\|$ should be close to zero when, given a suitable value of $\lambda$, a good regularized solution is computed (these are indeed the optimality conditions for problem (\ref{tikhx0})); therefore, a strict monitoring of the inexactness in $\whr\lamk$ may be necessary. 

We conclude this section by mentioning that inequalities (\ref{iLSQRbound}), (\ref{iCGLSbound}), (\ref{hiLSQRbound}) and (\ref{hiCGLSbound}), being expressed with respect to the $k$th projected solutions $s_k$ or $s_{\lamk}$, cannot be straightforwardly employed when, at the $j$th iteration of the inexact solvers, one may need to bound every $\|E_j\|$, $1\leq j\leq k$, to guarantee that $\|r_{k}\|$, $\|\whr_{k}\|$, $\|r{\lamk}\|$ and $\|\whr{\lamk}\|$ are close enough to their exact counterparts. 
To allow this, \cite{inexact1} considers upper bounds for the magnitude of the components of the $k$th projected solution 
depending on the norm of the $j$th inexact residual and the smallest singular value of the $k$th projected coefficient matrix: 
if the latter can be easily estimated ahead of the iterations, $\|E_j\|$ can then be adaptively bounded. This approach can be straightforwardly extended to the new methods introduced in this section. 
More precisely, denoting by $\sigma_k(C)$ the $k$th singular value of a matrix $C$, one can state that, if
\begin{equation}\label{iLSQRboundalt}
\|E_j\|\leq \frac{\sigma_k(M_k)}{k}\frac{1}{\|r_{j-1}\|}\eps\quad\mbox{and}\quad
\|E_j\|\leq \frac{(\sigma_k(M_k^TM_k+\lambda^2I))^{\nicefrac{1}{2}}}{k}\frac{1}{\|r_{\lambda,j-1}\|}\eps\,,
\end{equation}
then the last term in the last inequality in (\ref{iLSQRbound}) and (\ref{hiLSQRbound}), respectively, is bounded by $\eps$. 
%
Similar bounds can be derived for the iCGLS-based solvers (\ref{iCGLSbound}), (\ref{hiCGLSbound}). Although a careful analysis of the behavior of  $\sigma_k(M_k)$ and $\|r_{j-1}\|$ still has to be performed for iLSQR, it is well known that, if the (exact) GK algorithm were adopted in the framework of (\ref{eq:linsys}), then $\sigma_k(M_k)$ would eventually be numerically zero, while $\|r_{j-1}\|$ would eventually stabilize around $\|e\|$; see \cite{gazzola2016inheritance, survey, HnPlSt09}. This trend is enhanced in the case of severely ill-posed problems. If such a behavior is also assumed when using iGK, the first bound in (\ref{iLSQRboundalt}) would eventually prescribe a numerically zero $\|E_j\|$, while the second bound in (\ref{iLSQRboundalt}) would be more permissive, in that $(\sigma_k(M_k^TM_k+\lambda^2I))^{\nicefrac{1}{2}}$ would eventually stabilize around $\lambda$. A numerical illustration is given in Section \ref{ssec:solerror}.

\section{Inexact Krylov methods for blind deblurring}\label{sec:blind} 
In this section we explain how the inexact solvers presented in Section \ref{sec:inexact} can be adopted to solve separable nonlinear inverse problems of the form (\ref{eq:sepnonlinsys}), which we compactly rewrite as
\begin{eqnarray}
z_\lambda = \arg\!\!\!\!\min_{z\in\R^n, y\in\R^p}g(z,y)\,,\;\mbox{where}\; \begin{array}{lcl}
g(z,y)\!\!\!\!&=&\!\!\!\!\|F(z,y)\|^2\\
F(z,y)\!\!\!\!&=&\!\!\!\!\wtAlam(y)z - \wtr_0\\
\wtAlam(y)\!\!\!\!&=&\!\!\!\![\,A^T(y), \lambda I\,]^T,\;\wtr_0=[r_0^T, 0^T]^T\\
x_\lambda\!\!\!\!&=& x_0 + z_\lambda
\end{array}.
\label{eqtikhybis}
\end{eqnarray}
In particular, we will target blind image deblurring problems using a variable projection method, 
and we will display the results of some numerical tests, including comparisons with other Krylov-based approaches for blind deconvolution.

\subsection{Problem formulation}\label{ssec:pb}
Here and in the following, the unknown $x\in\R^n$ appearing in (\ref{eq:linsys}) is a vectorialized image obtained by stacking the columns of the 2D image $X\in\R^{N\times N}$, with $n=N^2$. 
The matrix $A\in\R^{n\times n}$ models a spatially invariant blurring process, i.e., a convolution process defined assigning a point spread function (PSF) that describes the deformation undergone by each entry (pixel) of $X$, and boundary conditions that prescribe the behavior of the pixels at the boundaries of $X$. Conventionally (see, e.g., \cite{JulianneJim}), 
a PSF $P\in\R^{N\times N}$ is a sparse 
image with only a few nonzero pixels located at the center of $P$. In the parametric model adopted here, the entries of the PSF are assigned an analytical expression depending on some parameters $y$: such a parametric PSF is denoted by $P(y)$. 
A parameter-dependent blurring matrix may be denoted by $A(P(y))$, using an alternative notation to $A(y)$ appearing in 
(\ref{eqtikhybis}); 
moreover, committing a slight abuse of notation, we may write matrix-vector products with $A$ as either $A(y)x$ or $A(y)X$. In the following we consider Gaussian blurs, where $y=[\sigma_1,\sigma_2,\rho]^T$ and the $(i,j)$th entry of the PSF centred at pixel $(\chi_1,\chi_2)$ reads
%
%
\begin{equation}\label{GaussPSF}
[P(y)]_{i,j}={c(\sigma_1,\sigma_2,\rho)}\exp\left(-\frac{1}{2}\left[\begin{array}{c}
i-\chi_1 \\
j-\chi_2
\end{array}\right]^T
\left[\begin{array}{cc}
\sigma_1^2 & \rho^2\\
\rho^2 & \sigma_2^2
\end{array}\right]^{-1}
\left[\begin{array}{c}
i - \chi_1\\
j - \chi_2
\end{array}\right]\right)\,.
\end{equation}
Here $\sigma_1$ and $\sigma_2$ determine the spread of the Gaussian, and $\rho$ determines its orientation; 
$c(\sigma_1,\sigma_2,\rho)$ is a scaling factor introduced so that $\sum_{i,j=1}^N[P(y)]_{i,j}=1$. Note that 
\begin{equation}\label{constry}
\sigma_1^2\sigma_2^2-\rho^4 >0
\end{equation}
should also be imposed for (\ref{GaussPSF}) to be meaningful. 
An important property of the blurring matrix is that, thanks to the particular structure of a blurring matrix $A$, 
\begin{equation}\label{blurdual}
A(y)X = A(y)x = A(P(y))x = A(X)P(y)\,.
\end{equation}
When considering blind deblurring within this framework, the parameters $y=y_\true$ defining the blur are unknown and should be recovered alongside an approximation of $x$. Without any loss of generality, reflexive (or Neumann) boundary conditions (fixed during the iterative solver) are assumed from now on. 

Up to Section \ref{ssec:numer}, we give some numerical illustrations of the behaviour of inexact solvers on a specific simple test problem generated as follows. We take the well-known \texttt{satellite} test image of size $256\times 256$ pixels from \cite{restoreT}: we apply a Gaussian blur (\ref{GaussPSF}) with parameters $y_\true=[2.5, 2.5, 0]^T$ (so that only one blurring parameter has to be recovered), followed by corruption by Gaussian white noise of level $\|e\|/\|b_\true\|=10^{-2}$. 
Exact and corrupted images, together with the exact PSF, are displayed in Figure \ref{fig:testP}. The quality of the reconstructions (for both $x$ and $y$) will be measured by the relative reconstruction error, i.e.,
\begin{equation}\label{RRE}
\mbox{RRE}_{x}=\frac{\|x-x_\true\|}{\|x_\true\|},\quad \mbox{RRE}_{y}=\frac{\|y-y_\true\|}{\|y_\true\|}
\end{equation}

Exploiting the fact that (\ref{eqtikhybis}) is linear in $z=x-x_0$, the variable projection method \cite{GolubVP} implicitly eliminates the dependence on the linear parameter $z$, and obtains a reduced cost functional that depends on $y$ only. More precisely, using the same notations as in (\ref{eqtikhybis}), we introduce the functional 
\begin{equation}\label{hdef}
h(y):=g(z_\lambda(y),y)\,,\quad\mbox{where}\quad \begin{array}{lcl}
z_\lambda(y) &=& \arg\min_{z\in\R^n}g(z,y)\\
&=&(A^T(y)A(y)+\lambda^2I)^{-1}A^T(y)r_0
\end{array},
\end{equation}
and take $x_\lambda(y)=x_0+z_\lambda(y)$. 
We then apply Gauss-Newton to minimize $h(y)$, so that we have to compute the gradient
\begin{equation}\label{gradh}
\nabla_y h(y)=\frac{dz_\lambda}{dy}\nabla_{z_{\lambda}} g(z_\lambda,y)+\nabla_{y} g(z_\lambda,y)=\nabla_{y} g(z_\lambda,y)=J_h^TF(z_\lambda,y)\,,
\end{equation}
where $J_h$ is the Jacobian of the function $F$ defined in (\ref{eqtikhybis}), i.e., \vspace{-0.2cm}
\begin{equation}\label{def:Jacob}
J_h = \left[\begin{array}{c}
\nicefrac{d (A(y)z_\lambda)}{d y}\\
0
\end{array}
\right]= \left[\begin{array}{c}
\widehat{J}_h\\
0
\end{array}
\right]\,.
\vspace{-0.1cm}
\end{equation}
In deriving (\ref{gradh}) we used the chain rule and, in the penultimate equality, the fact that $\nabla_{z_\lambda}g(z_\lambda,y)=0$ because of the definition of $z_{\lambda}(y)$ in (\ref{hdef}). The main steps involved in the application of the Gauss-Newton method to minimize $h(y)$ in (\ref{hdef}) are summarized in Algorithm \ref{alg:old}, lines 7 to 9. As observed in \cite{JulianneJim}, the Jacobian $\widehat{J}_h\in\R^{p\times n}$ can be computed analytically exploiting the property (\ref{blurdual}) and, since $p\ll n$, the LS problem in line 8 of Algorithm \ref{alg:old} can be conveniently solved. The steplength $\gamma_l$ in line 9 of Algorithm \ref{alg:old} can be determined using a line search (such as an Armijo rule), which may require a repeated computation of $z_{\lambda}(y)$ (see, e.g., \cite{WrightOpt}). 
We emphasise that, in the setting of unstructured large-scale problems, two main challenges arise: first, $z_\lambda$ cannot be computed directly using the formula appearing in (\ref{hdef}); second, a suitable value of the regularization parameter $\lambda$ may not be known in advance of the iterations and may depend on the current value of $y$. The authors of \cite{JulianneJim} devise an efficient and effective way of overcoming these challenges by using the LSQR-based hybrid method (\ref{eq:projregpb}), 
with adaptive regularization parameter choice and a stopping criterion based on GCV: this is summarized in Algorithm \ref{alg:old}, lines 3 to 6. 
%
\begin{algorithm}
\caption{Variable projection with Gauss-Newton and \emph{hybrid-LSQR} solver}
\label{alg:old}
\begin{algorithmic}[1]
\State{Choose initial guesses $x_0$ and $y_0$.}
\For{$l=1,2,\dots$ until a stopping criterion is satisfied}
\For{$k=1,2,\dots$ until a stopping criterion is satisfied}
\State{Expand $\Kry_k(A(y_{l-1})^TA(y_{l-1}), A(y_{l-1})^Tr_0)$} using GK (\ref{eq:Lanczos1})
\State{Compute $x\lamk$ solving problem (\ref{eq:projregpb}) with adaptive choice of $\lambda$}
\EndFor
\State{Compute the residual ${r}_{l-1} = b-A(y_{l-1})x\lamk$}
\State{Compute $d_{l-1} = \arg\min_d\|\widehat{J}_hd-{r}_{l-1} \|$}
\State{Update $y_{l}=y_{l-1}+\gamma_l d_{l-1}$ (setting the steplength $\gamma_l$)}
\State{Update $x_0$}
\EndFor
\end{algorithmic}
\end{algorithm}

\subsection{Solution by inexact Krylov methods and error control}\label{ssec:solerror}
The method outlined in Algorithm \ref{alg:old} involves an inner-outer iteration scheme, where 
a hybrid-LSQR method fully runs for each value of the blurring parameters determined within the Gauss-Newton outer iterations. The basic idea leading to the use of inexact Krylov methods in the setting of blind deblurring is to allow Gauss-Newton updates of the blurring parameters at each iteration of 
the hybrid method used to approximate the deblurred image. This implies that the coefficient matrix for the computation of $z_\lambda(y)$ in (\ref{eqtikhybis}) is applied with varying amount of inexactness, using the hybrid-iLSQR or the hybrid-iCGLS methods (Section \ref{sec:hybridiLSQR} and \ref{sec:hybridiCGLS}, respectively) rather than the hybrid-LSQR. 
The bounds derived in Section \ref{sec: theory} should be employed to monitor the quality of the solution: when exceeding the tolerated amount of inexactness, the hybrid inexact methods should be restarted. 
In the following we explain how inexactness is defined in the blind deblurring setting, and we tailor the iGK algorithm 
to this application; a sketch is provided in Algorithm \ref{alg:new}. Note that only the hybrid-iLSQR method will be considered from now on: the derivations below can be easily extended to hybrid-iCGLS, which performs very similarly to hybrid-iLSQR on the tested problems. 

\begin{figure}
\begin{tabular}{ccc}
\hspace{-1.6cm}{\small {\bf (a)} $x_\true$} & \hspace{-1.6cm}{\small {\bf (b)} $y_\true$} & \hspace{-1.6cm}{\small {\bf (c)} $b$}\vspace{-0.0cm}\\
\hspace{-1.6cm}\includegraphics[width=5.5cm]{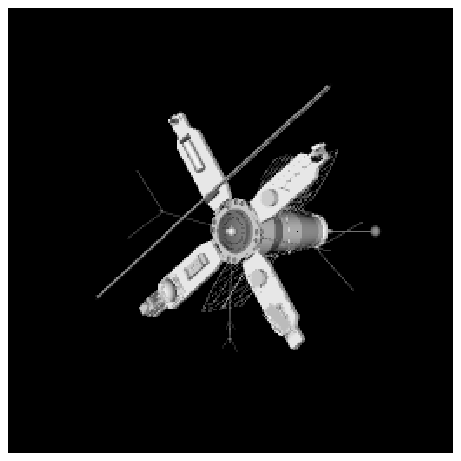} &
\hspace{-1.6cm}\includegraphics[width=5.5cm]{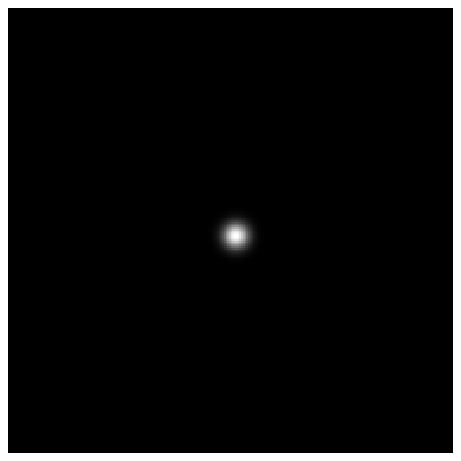} & 
\hspace{-1.6cm}\includegraphics[width=5.5cm]{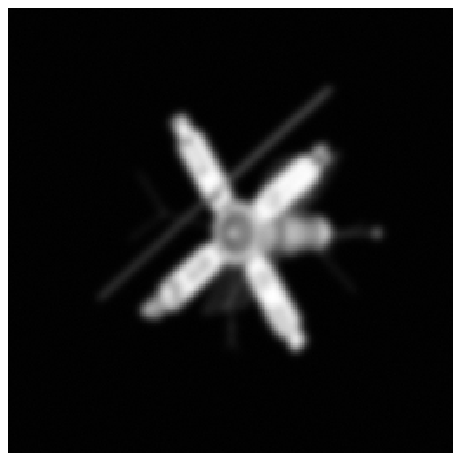}\vspace{-0.7cm}\\
\end{tabular}
\caption{Illustrative \texttt{satellite} test problem. {\bf (a)} Exact test image $x_\true$. {\bf (b)} Blow up (400\%) of the exact Gaussian PSF (\ref{GaussPSF}) with $y_\true=[2.5,2.5,0]^T$. 
{\bf (c)} Data $b$ (blurred and noisy version of $x_\true$, with $\|e\|/\|b_{\true}\|=10^{-2}$). 
}\label{fig:testP}
\vspace{-0.5cm}
\end{figure}

\begin{algorithm}
\caption{Variable projection with Gauss-Newton and \emph{hybrid-iLSQR} solver}
\label{alg:new}
\begin{algorithmic}[1]
\State Choose initial guesses $x_0$ and $y_0$; set an accuracy $\eps$.
\For{$l=1,2,\dots$ until a stopping criterion is satisfied}
\For{$k=1,2,\dots$ until inexactness (bounded by $\eps$) is tolerated}
\State{Expand the approximation subspace $\mathcal{R}(V_k)$ using $A(y_{k-1})$ and iGK (\ref{eq:inex_gkb_2})}
\State{Compute $x\lamk$ solving problem (\ref{eq:iprojregaugpb_1}) with adaptive choice of $\lambda$}
\State{Compute the residual ${r}_{k-1} = b-A(y_{k-1})x\lamk$}
\State{Compute $d_{k-1} = \arg\min_d\|\widehat{J}_hd-{r}_{k-1} \|$}
\State{Update $y_{k}=y_{k-1}+\gamma_kd_{k-1}$ (setting the steplength $\gamma_k$)}
\EndFor
\State{Update $x_0$; take $y_0=y_k$}
\EndFor
\end{algorithmic}
\end{algorithm}

Defining inexactness in the setting of blind deblurring is not straightforward: since $y_\true$ is unknown, $A(y_\true)$ is unavailable and an absolute concept of error cannot be defined. Our pragmatical point of view is to consider as exact blurring matrix the latest computed approximation of $A(y)$. That is, if $j-1$ iterations of Algorithm \ref{alg:new} (lines 3 to 9) are performed, $A(y_{j-1})$ is regarded as the exact coefficient matrix for the $j$th iteration, so that 
\begin{equation}\label{def:error}
A(y_{i-1}) = A(y_{j-1}) + E_i^j\,,\quad\mbox{where}\quad E_i^j:=A(y_{i-1}) - A(y_{j-1})\,,
\end{equation}
is the inexact matrix applied at iteration $i\leq j$. A few remarks are in order here. First of all, such a definition of inexactness is iteration-dependent, i.e., it is valid at the $j$th 
iteration of Algorithm \ref{alg:new} only, and should be updated as the iterations proceed; hence the notation $E_i^j$ for the error in $A(y_{i-1})$. Second, according to (\ref{def:error}), when computing the $j$th product at the $j$th iGK iteration, $E_{j}^j=0$: this is somewhat opposite to the situation described in \cite{inexact1}, where an increasing amount of inexactness is generally allowed as the iterations of the inexact solver proceed. 
The inexactness defined in (\ref{def:error}) can anyway be tolerated, as far as the bounds derived in Section \ref{sec: theory} can be satisfied. 
Finally, the definition of inexactness (\ref{def:error}) well-matches with the approach adopted in Algorithm \ref{alg:old}: indeed, given a current value $y_{j-1}$ of the blurring parameters, $E_i^j$ quantifies how much the previous $y_{i-1}$, $i\leq j$, are allowed differ, so that the performance of hybrid-iLSQR with varying $y$ in Algorithm \ref{alg:new} is similar to the performance of hybrid-LSQR with fixed $y=y_{j-1}$ in Algorithm \ref{alg:old}. 

According to the definition (\ref{def:error}), after $j$ iGK iterations are performed 
(leading to the blurring parameter $y_{j}$), 
the partial decomposition in (\ref{eq:inex_gkb_2}) reads 
\[
(A(y_{j-1})+\mathcal{E}_j^j)V_j = U_{j+1} M_j\,,\quad
(A(y_{j-1})+{\mathcal{F}}_{j+1}^j)^T U_{j+1}= V_{j+1} L_{j+1}^T\,,
\]
where
\begin{equation}\label{errblind}
\begin{array}{lcl}
\mathcal{E}_j^j &=&\sum_{i=1}^{j} E_{i}^j v_i v_i^T\\
\mathcal{F}_{j+1}^j&=&\sum_{i=1}^{j+1} (E_{i-1}^j)^T u_iu_i^T\\
\end{array}
\quad\mbox{and}\quad
\begin{array}{l}
E_0^j = A(y_0) - A(y_{j-1})\\
E_i^j\; \mbox{ is as in (\ref{def:error}), $i=1,\dots,j$}
\end{array}.
\end{equation}
The specific expressions (\ref{errblind}) are linked to the fact that the updated $y_j$ is computed after both $u_{j+1}$ and $v_{j+1}$ are computed (in this order). 

To guarantee that, at the $j$th iteration of Algorithm \ref{alg:new} with a fixed regularization parameter $\lambda$, the norm of the `exact' residual (which would have been obtained applying hybrid-iLSQR with coefficient matrix $A(y_{j-1})$) is sufficiently close to the norm of the computed residual, the bound (\ref{hiLSQRbound}) should be checked, using the iteration-specific definition of $E_i^j$, $i=1,\dots,j$ in (\ref{def:error}): the $(j+1)$th iGK iteration is performed only if such a condition is satisfied (and condition (\ref{hiLSQRbound}) is then checked with the updated $E_i^j$, $i=1,\dots,j+1$); otherwise iGK should be restarted, possibly taking as initial guess $x_0$ for the deblurred image the last valid approximation of $x$, i.e., $x_0=x_{\lambda,j-1}$, and as an initial guess $y_0$ for the blurring parameters their last computed value, i.e., $y_0=y_j$. Alternatively, to guarantee that `exact' and computed residuals are sufficiently close, one can employ the second bound in (\ref{iLSQRboundalt}). It should be stressed that, in the blind deblurring setting, the amount of inexactness is dictated by the Gauss-Newton updates and, therefore, cannot be adaptively set. 

Figure \ref{fig:bounds} displays the behavior of some relevant quantities obtained running 60 iLSQR and hybrid-iLSQR iterations (the latter with a fixed regularization parameter $\lambda=5\cdot 10^{-1}$), starting with $x_0=0$ and $y_0=[7,7,0]^T$. 
Looking at frame (a) we can clearly see that, as the number of iGK iterations $k$ increases, both $\|r_k\|$ and the smallest singular value of the iLSQR projected matrix $M_k$ steadily decrease; when considering hybrid-iLSQR, thanks to regularization, $\|r\lamk\|$ stabilizes and the decay of the smallest singular value is slower. This implies that the bounds in (\ref{iLSQRboundalt}), for a fixed $\eps$, are more strict in the iLSQR than in the hybrid-iLSQR case, as it is evident in frames (c) and (d) (where $\eps=1$). Because of the behavior of $\|r_{j-1}\|$ and $\|r_{\lambda,j-1}\|$, the most stringent bound in (\ref{iLSQRboundalt}) is the one for $j=0$: this is depicted in frame (b) for values of $k=1,\dots,60$. 
\begin{figure}
\begin{tabular}{cc}
{\small \bf (a)} & {\small \bf (b)}\vspace{-0.1cm}\\
\includegraphics[width=4.7cm]{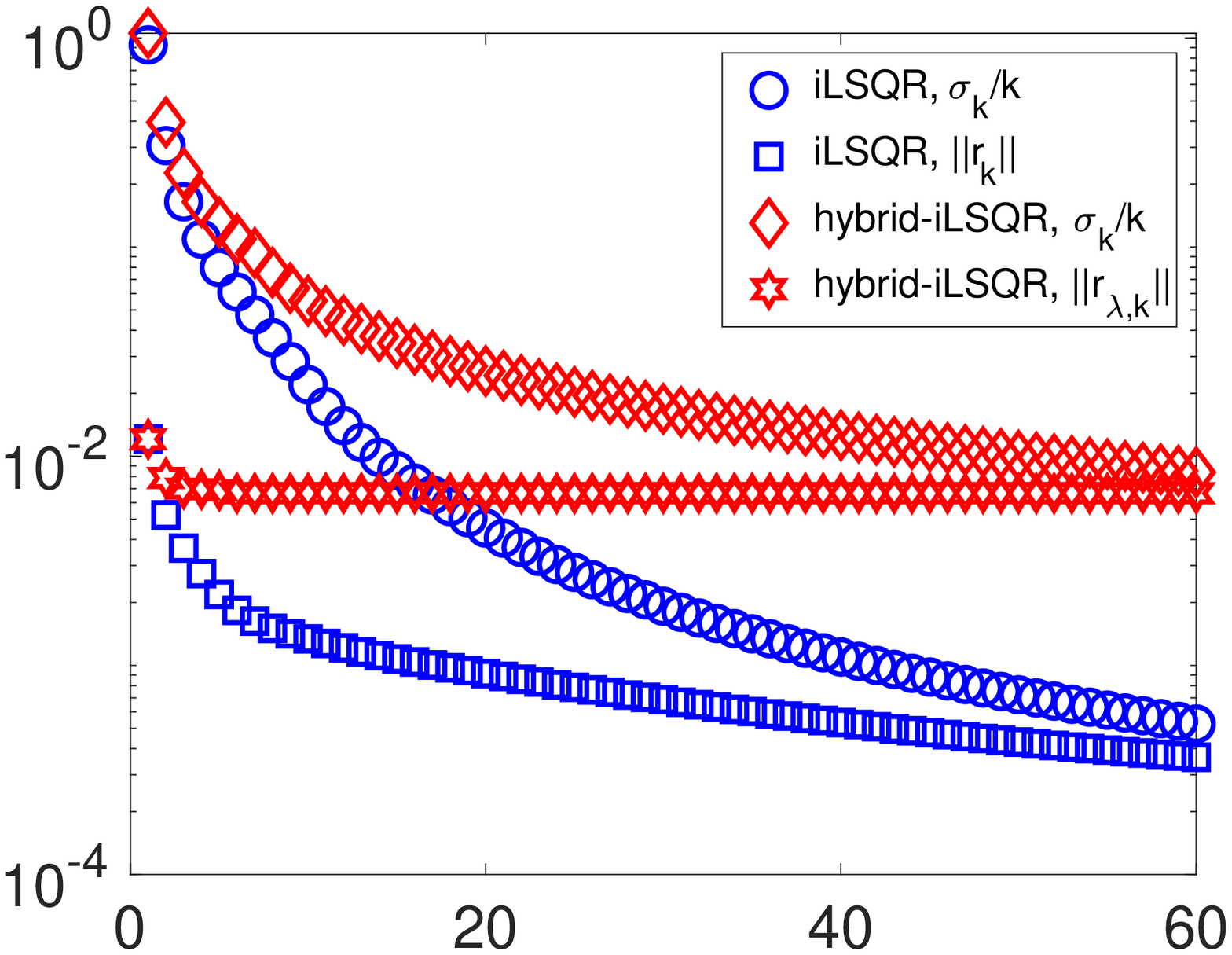} &
\includegraphics[width=4.7cm]{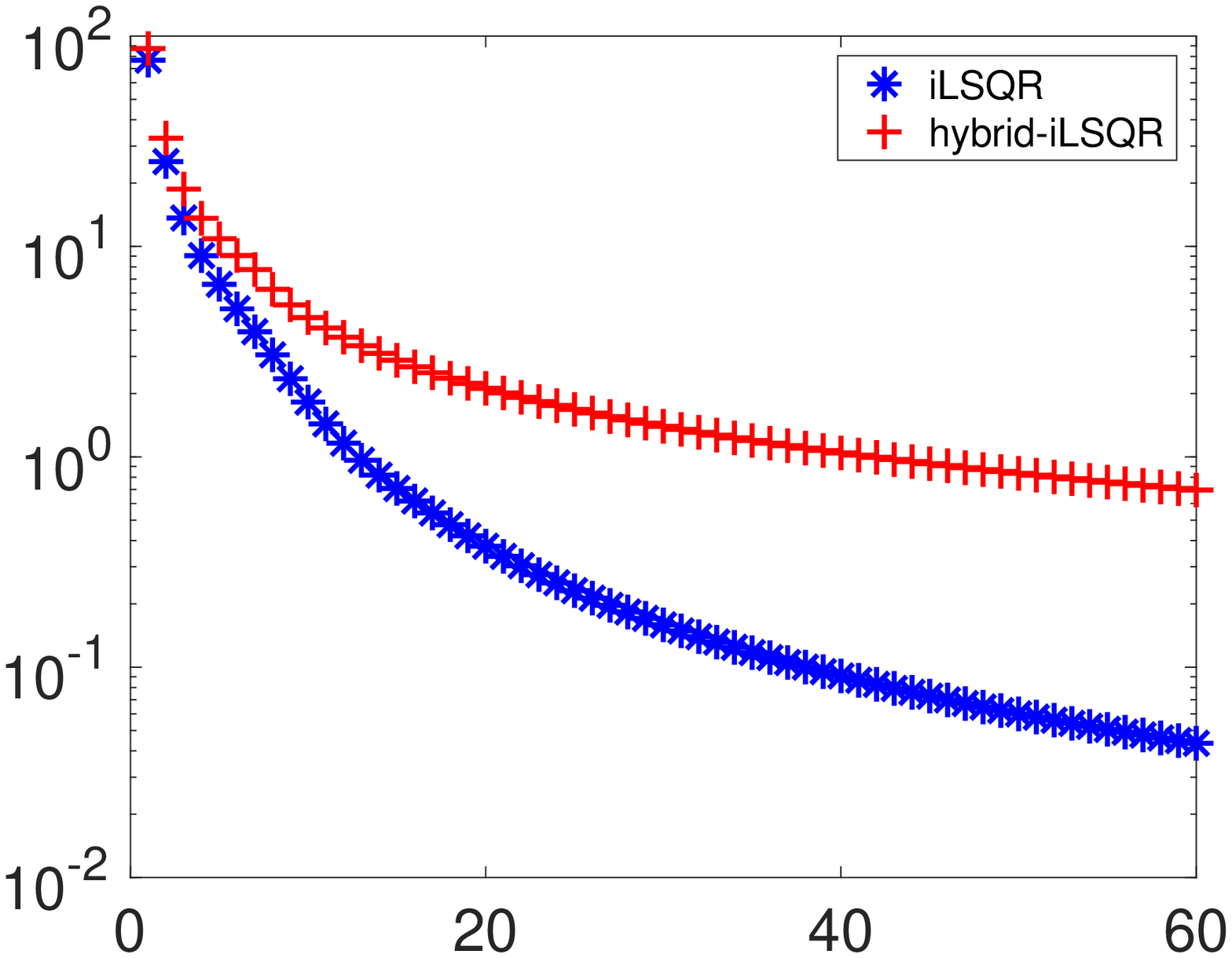}\vspace{-0.2cm}\\
{\small \bf (c)} & {\small \bf (d)}\vspace{-0.1cm}\\
\includegraphics[width=4.7cm]{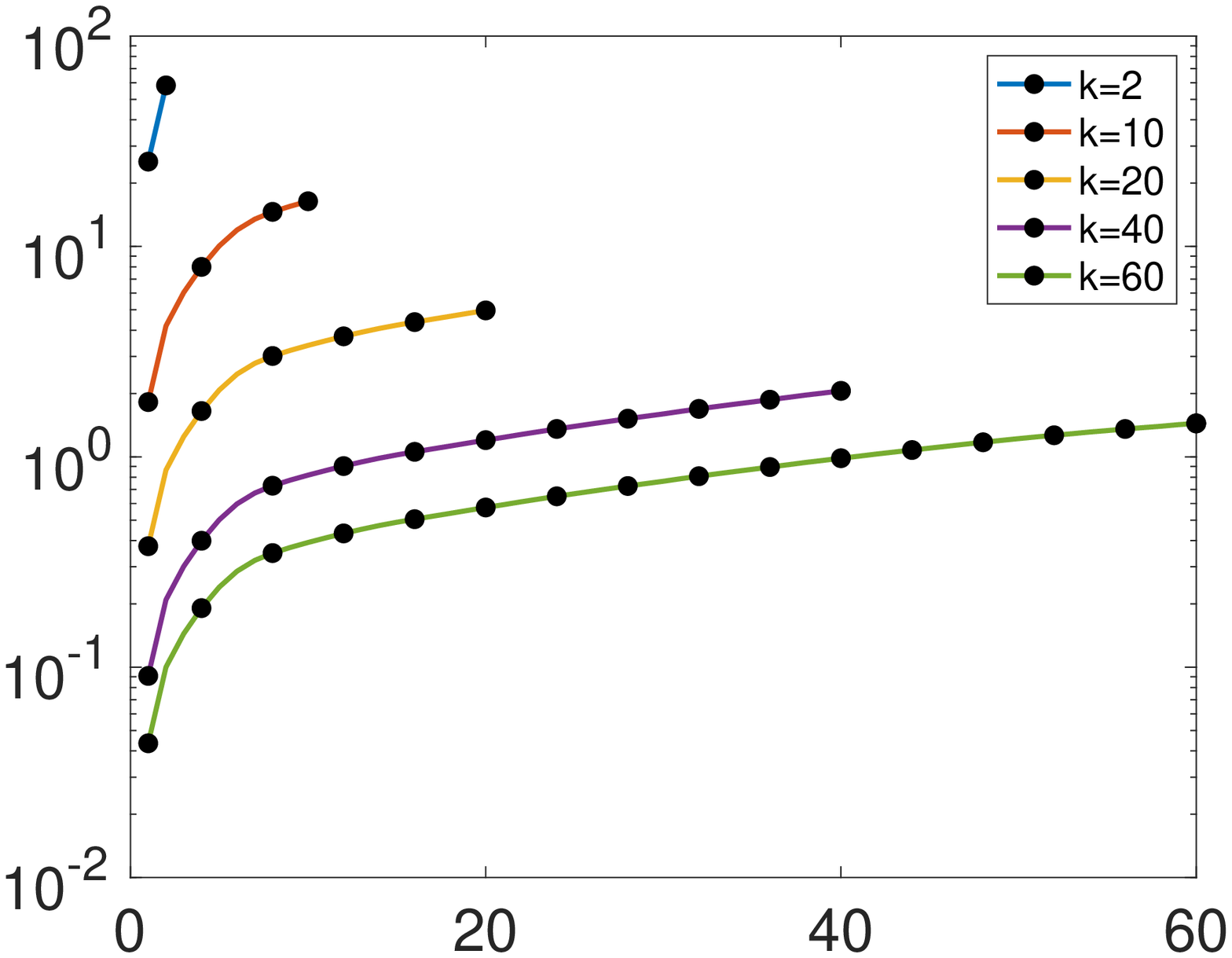} &
\includegraphics[width=4.7cm]{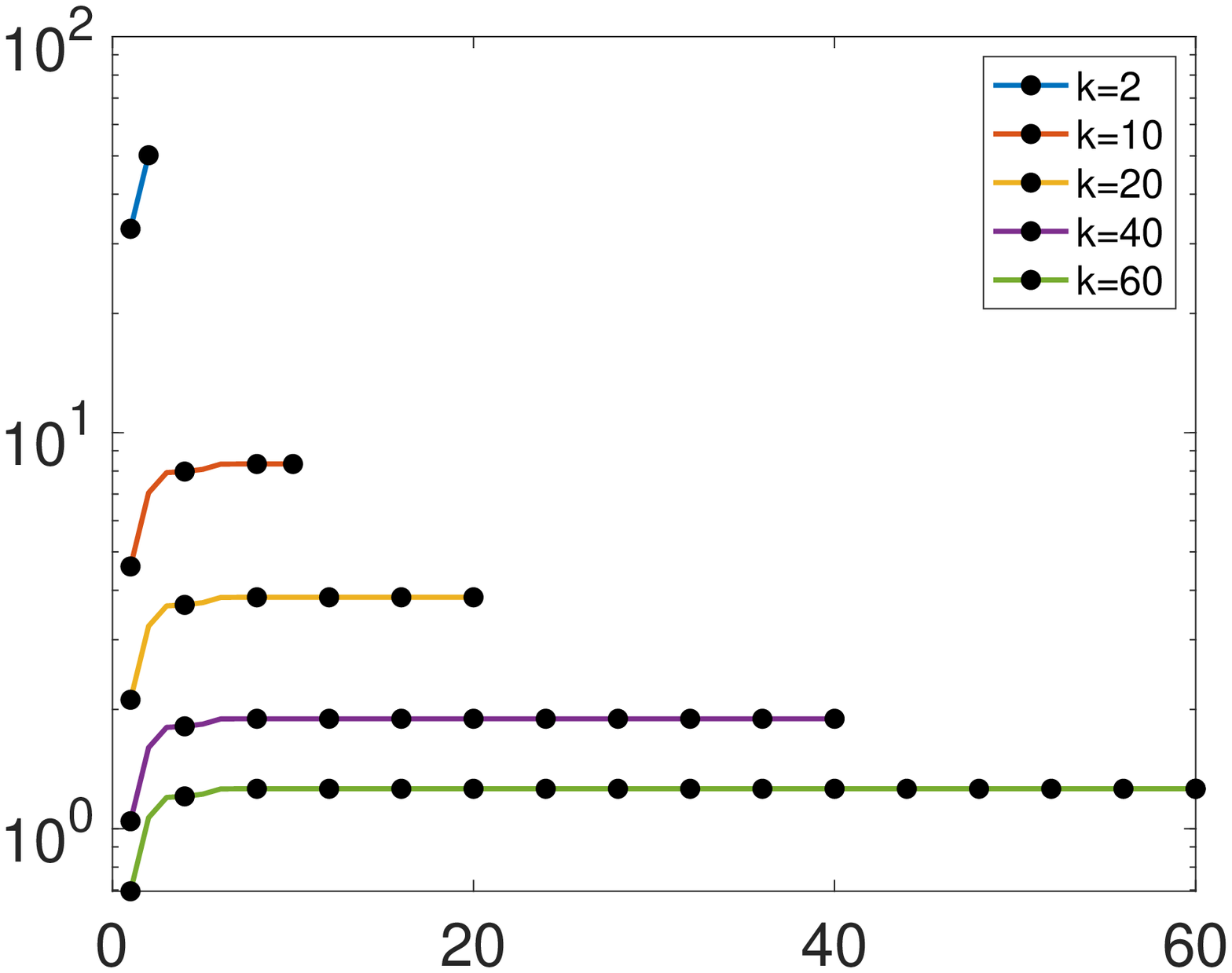}
\end{tabular}
\caption{Illustrative \texttt{satellite} test problem, with $\lambda=5\cdot 10^{-1}$: quantities appearing in the bounds (\ref{iLSQRboundalt}) after 60 iGK iterations are performed. {\bf (a)} History of the quantities $\sigma_k(M_k)/k$ and $\|r_{k-1}\|$ (for iLSQR) and $\sigma_k((M_k^TM_k+\lambda^2I))^{1/2}/k$ and $\|r_{\lambda,k-1}\|$ (for hybrid-iLSQR) versus $k$. {\bf (b)} History of the quantities $\sigma_k(M_k)/(k\|r_{0}\|)$ (for iLSQR) and $\sigma_k((M_k^TM_k+\lambda^2I))^{1/2}/(k\|r_{\lambda,0}\|)$ (for hybrid-iLSQR) versus $k$. {\bf (c)} Bounds (\ref{iLSQRboundalt}) for iLSQR: history of the quantities $\sigma_k(M_k)/(k\|r_{j-1}\|)$ versus $j$, for $k=2,10,20,40,60$ and $j=1,\dots,k$. {\bf (d)} Bounds (\ref{iLSQRboundalt}) for hybrid-iLSQR: history of the quantities $\sigma_k((M_k^TM_k+\lambda^2I))^{1/2}/(k\|r_{\lambda,j-1}\|)$ versus $j$, $j=1,\dots,k$, $k=2,10,20,40,60$.\vspace{-0.5cm}}\label{fig:bounds}
\end{figure}

We conclude this section by providing some details about the blurring parameters updates performed by the Gauss-Newton method (Algorithm \ref{alg:new}, line 8), as the inexact solvers used to approximate $z_\lambda$ in (\ref{eqtikhybis}) also have an impact on computation of $y$. 
Similarly to what happens at the $l$th outer iteration of Algorithm \ref{alg:old}, at the $(j+1)$th iteration of Algorithm \ref{alg:new} we would like the (exact) Tikhonov objective function (\ref{eqtikhybis}) to decrease, i.e., 
\begin{equation}\label{idealineqy}
\|\wtAlam(y_j)z_{\lambda,j+1}-\wtr_0\|\leq 
\|\wtAlam(y_{j-1})z_{\lambda,j}-\wtr_0\|\,.
\end{equation}
Assuming that the regularization parameter $\lambda$ is fixed and using the notations
\[
\wtE^{j}_i=[\,(E^j_i)^T,\,0^T\,]^T\in\R^{2n\times n},\; i=0,\dots,j,\quad \wtmcE^{j}_j=[\,(\mcE^j_j)^T,\,0^T\,]^T\in\R^{2n\times n}\,,
\]
where $E^j_i$ and $\mcE_i^j$ are defined as in (\ref{errblind}), let us assume that
\begin{equation}\label{ass:boundE}
\|\wtE_0^{j+1} x_0\| + \|\wtmcE_{j+1}^{j+1}\underbrace{V_js_{\lambda,j}}_{=z_{\lambda,j}}\|\leq \wteps\,,\quad \|\wtE_0^{j+1} x_0\| + \|\mcE_{j+1}^{j+1}\underbrace{V_{j+1}s_{\lambda,j+1}}_{=z_{\lambda,j+1}}\|\leq \wteps\,.
\end{equation}
%
Denoting by $\wts_{\lambda,j+1}=[s_{\lambda,j}^T, 0]^T\in\R^{j+1}$, it follows that 
\begin{eqnarray*}
& & \|\wtAlam(y_{j})V_{j+1}s_{j+1} - \wtr_0\|-\wteps \leq \|(\wtAlam(y_{j})+\wtmcE_{j+1}^{j+1})V_{j+1}s_{j+1}- \wtr_0\|\\
&\leq& \|(\wtAlam(y_{j})+\wtmcE_{j+1}^{j+1})V_{j+1}\wts_{\lambda,j+1}- \wtr_0\| =  \|(\wtAlam(y_{j})+\wtmcE_{j+1}^{j+1})V_{j}s_{\lambda,j}- \wtr_0\| \\
&\leq& \|\wtAlam(y_{j})V_{j}s_{\lambda,j}- \wtr_0\|+\wteps
\leq \|\wtAlam(y_{j-1})V_{j}s_{\lambda,j}- \wtr_0\|+\wteps\,.
\end{eqnarray*}
In the above chain of inequalities, the first one comes from the triangular inequality and (\ref{ass:boundE}), the second one comes from the hybrid-iLSQR optimality property (\ref{opthiLSQR}), the third equality holds because $E_{k+1}^{k+1}=0$, the fourth inequality comes from the triangular inequality and (\ref{ass:boundE}) and, finally, the fifth inequality holds because of the Gauss-Newton step (lines 7 and 8 of Algorithm \ref{alg:new}, with $J_h$ defined as in (\ref{def:Jacob})). 
Therefore, instead of (\ref{idealineqy}), we get
\begin{equation}\label{trueineqy}
\|\wtAlam(y_j)z_{\lambda,j+1}-\wtr_0\|\leq 
\|\wtAlam(y_{j-1})z_{\lambda,j}-\wtr_0\|+2\wteps\,,
\end{equation}
so that, because of the inexactness in hyrbid-iLSQR, in theory the decrease of the objective function in (\ref{eqtikhy}) is not guaranteed. To mitigate the theoretical lack of monotonicity (\ref{trueineqy}), we choose an optimal steplength in the Gauss-Newton step with respect to the second argument of $g(z_\lambda(y),y)$ only, i.e., we compute
\begin{equation}\label{yupdateapprox}
y_{j}= y_{j-1} + \gamma_j d_{j-1}\,,\quad\mbox{where}\quad
\gamma_{j}=\arg\min_{\gamma\geq 0}g(z_\lambda(y_{j-1}),y_{j-1} + \gamma d_{j-1})\,.
\end{equation}
This can be achieved using a numerical optimizer, such as MATLAB's \texttt{fminsearch}, around $\gamma=1$. 
Figure \ref{fig:theoryill} frame (e) displays the behavior of the exact and inexact versions of the objective function $g$ appearing in (\ref{eqtikhybis}), i.e., computed with $\wtAlam(y_{j})$ and $(\wtAlam(y_{j})+\wtmcE_{j+1}^{j+1})$, respectively: one can clearly see that both of them decrease as the iterations $j$ progress, and their values are quite similar. Although these results refer to the test problem described in Section \ref{ssec:pb} with $\wteps=10^{-2}$, this desirable behavior was observed in all the performed experiments.
Finally, in order for the Gaussian PSF (\ref{GaussPSF}) to be defined, the entries of $y$ should satisfy the constraint (\ref{constry}), which should be imposed when computing the Gauss-Newton update (\ref{yupdateapprox}) at the $j$th iteration of Algorithm \ref{alg:new}. Moreover, as the left-hand-side of (\ref{constry}) approaches zero, the PSF reduces to a single bright pixel, and multiplication by the corresponding blurring matrix (at the $(j+1)$th iteration of Algorithm \ref{alg:new}) would lead to stagnation of the iGK algorithm (\ref{eq:iGKDvect}). When constraint violation happens we skip the Gauss-Newton update of $y$ (i.e., we take $\gamma_j=0$) but we keep updating $x$, leading to a reduction of the inexact objective function in (\ref{eqtikhybis}). 
\begin{figure}
\begin{tabular}{ccc}
\hspace{-1.0cm}{\small \bf (a) $\RRE_x$} & \hspace{-0.7cm}{\small \bf (b) $\sigma_1$} & \hspace{-0.7cm}{\small \bf (c) $(V_{10}^{\lambda_1})^TV_{10}^{\lambda_2}$}\vspace{-0.0cm}\\
\hspace{-1.0cm}\includegraphics[width=4.7cm]{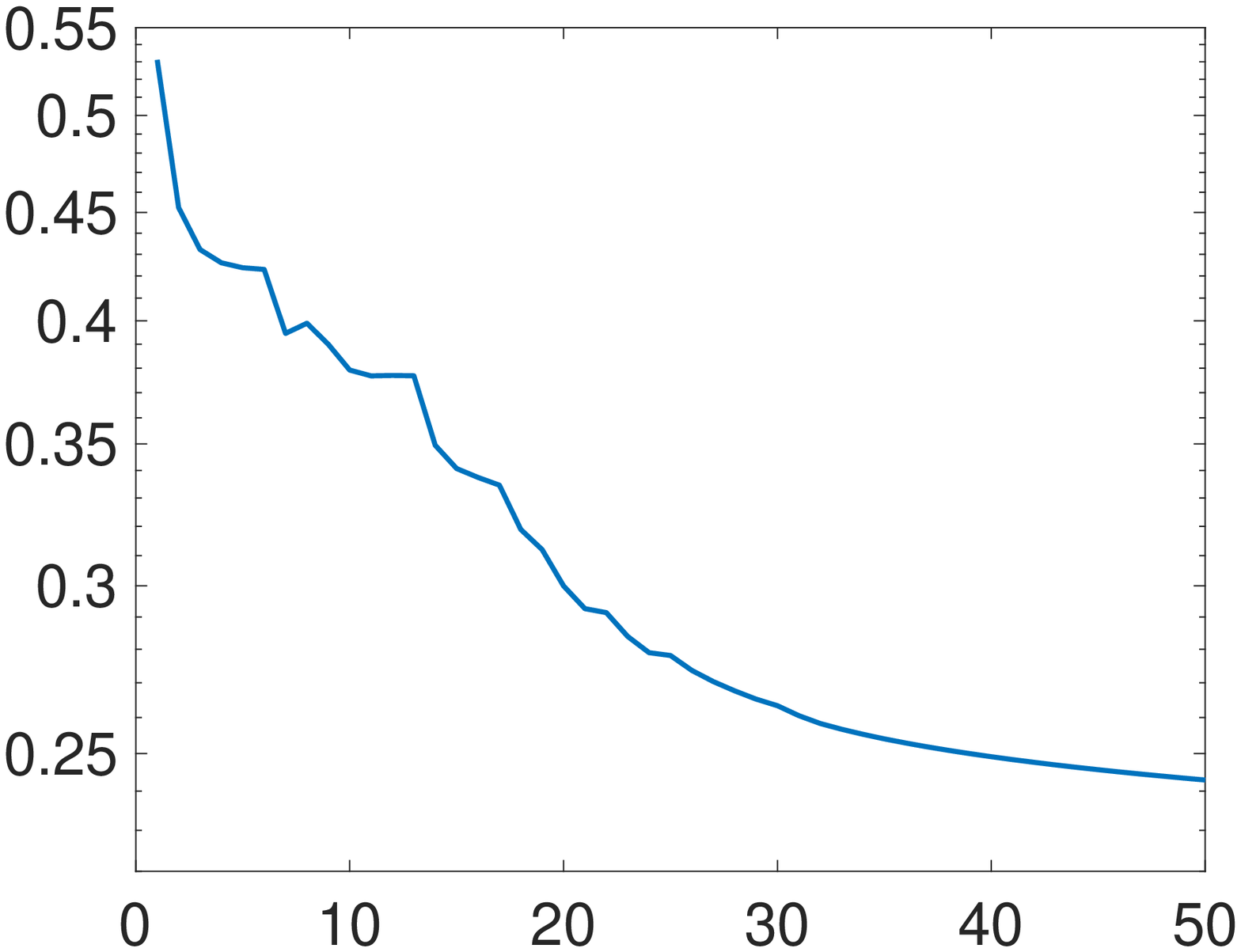} &
\hspace{-0.7cm}\includegraphics[width=4.7cm]{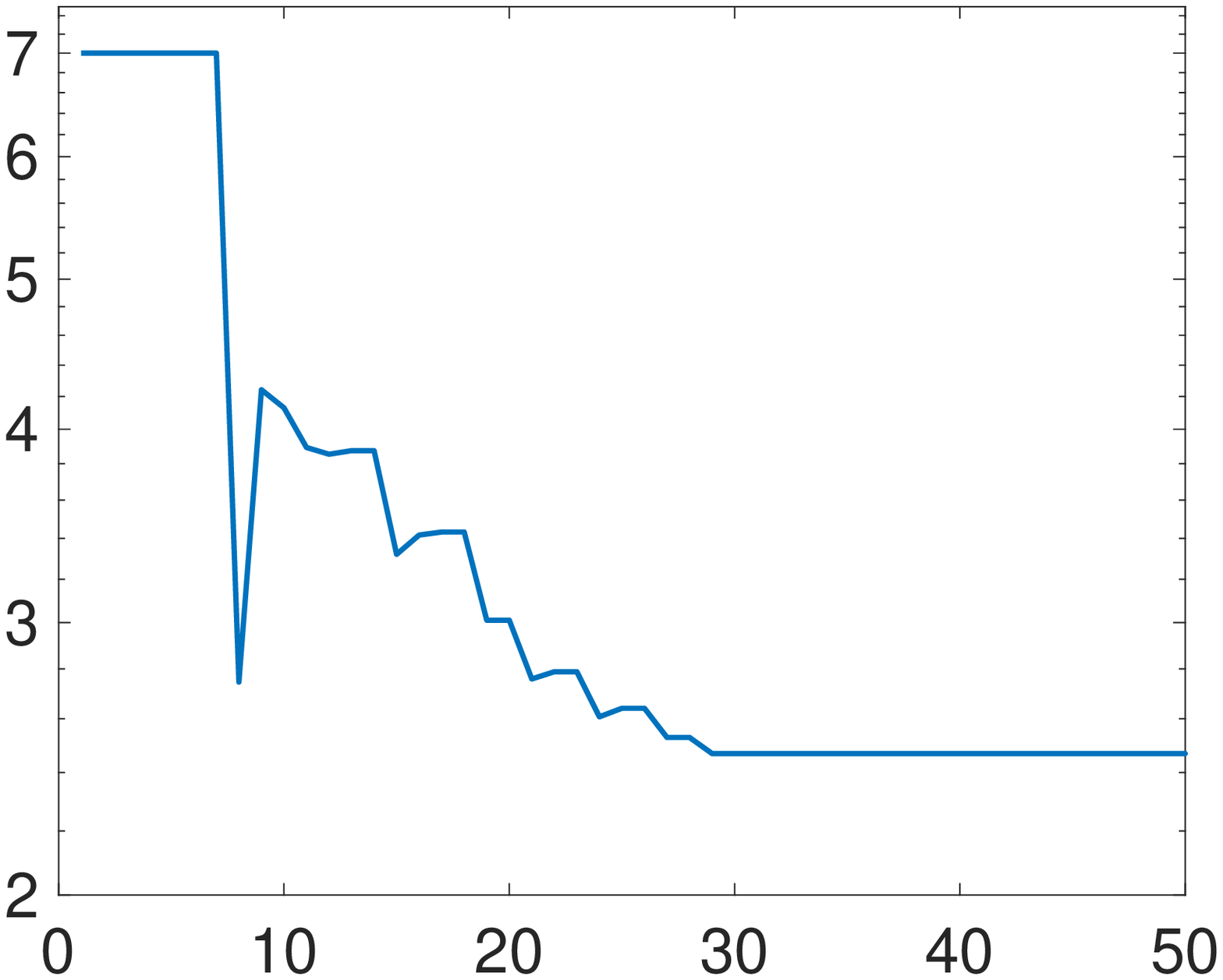} & 
\hspace{-0.7cm}\includegraphics[width=4.7cm]{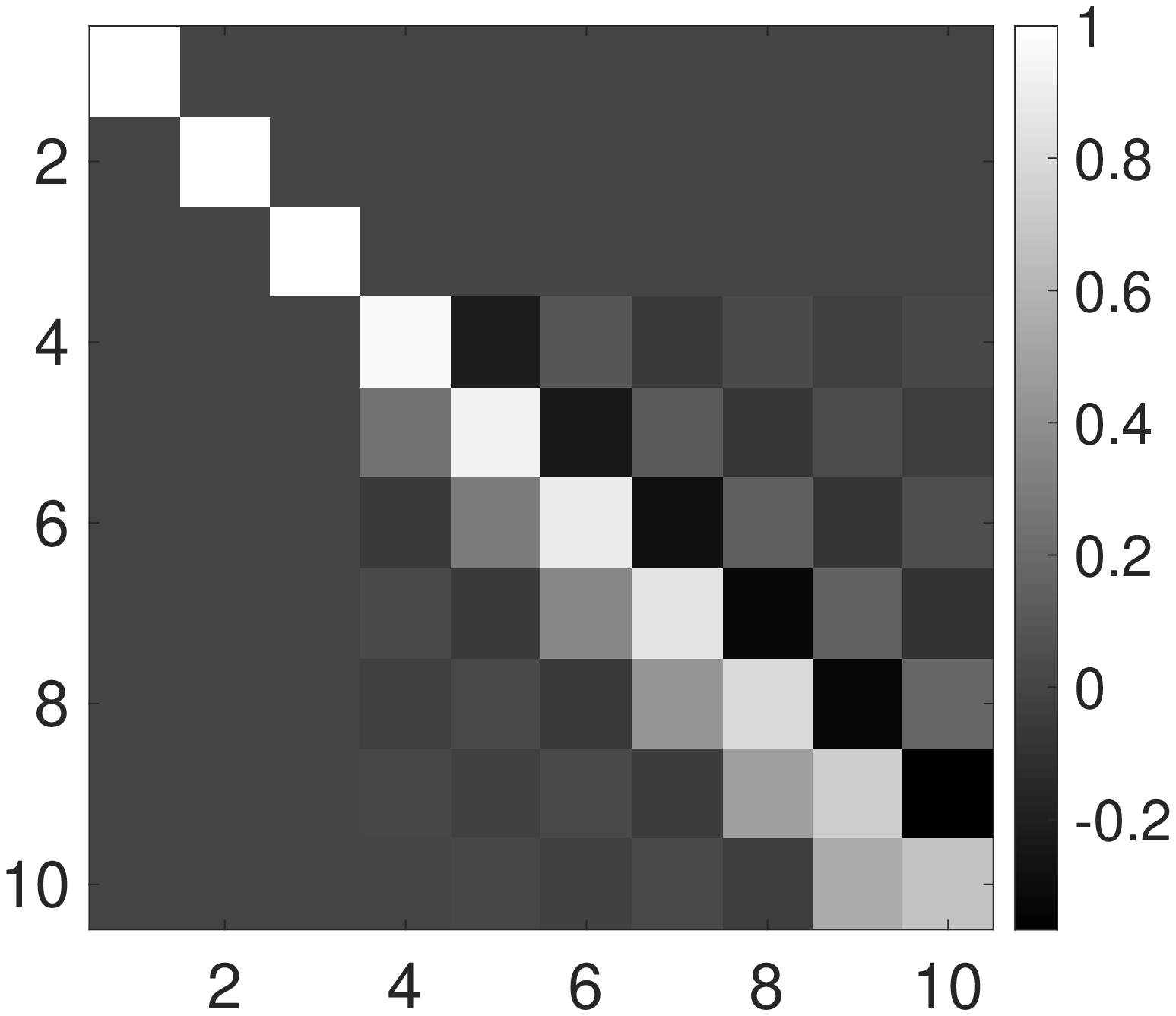}\vspace{-0.1cm}\\
\hspace{-1.0cm}{\small \bf (d) iteration count} & \hspace{-0.7cm}{\small \bf (e)} & \hspace{-0.7cm}{\small \bf (f) $(V_{10}^{\lambda_1})^TV_{10}^{\lambda_2}$}
\vspace{-0.0cm}\\
\hspace{-1.0cm}\includegraphics[width=4.7cm]{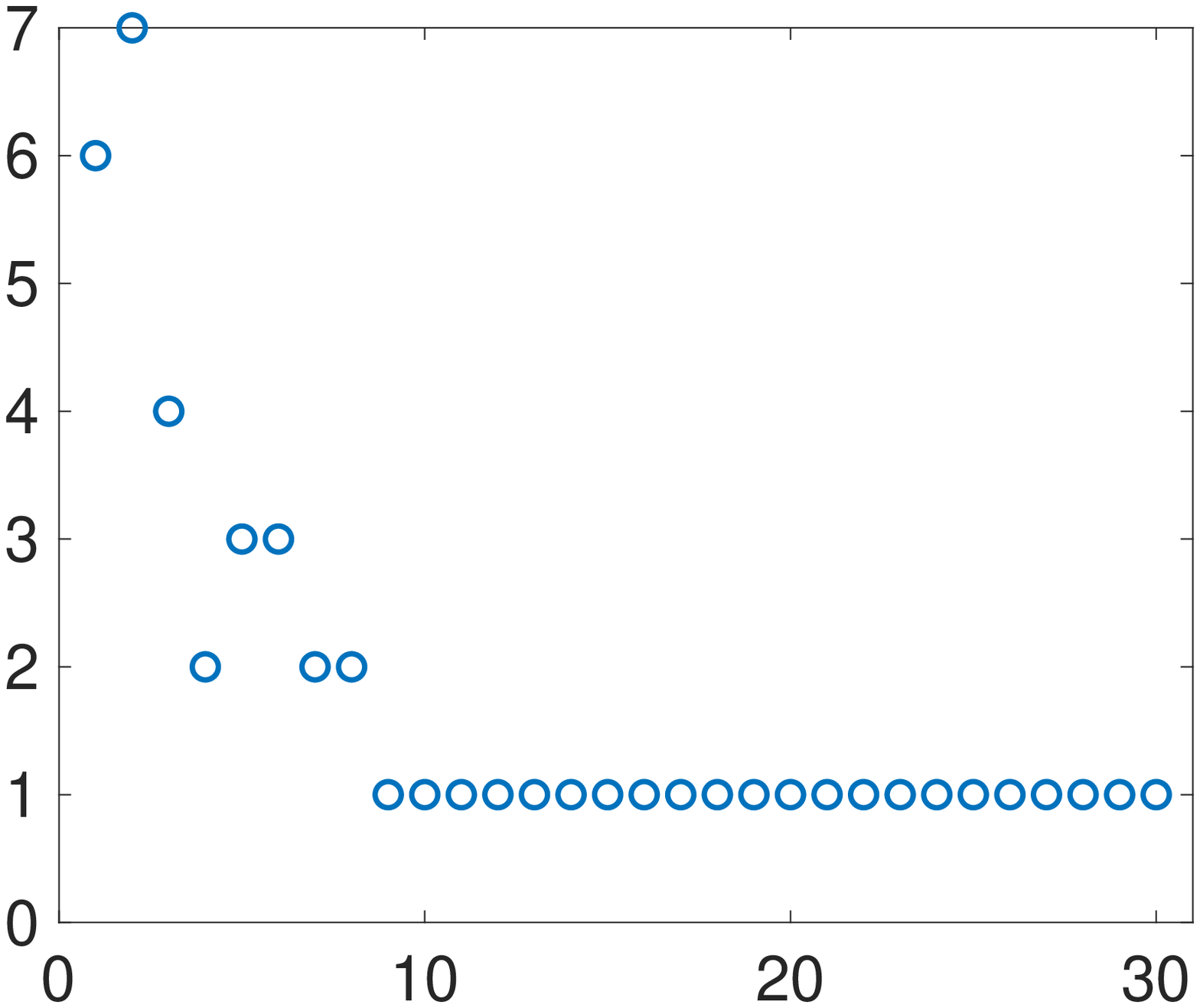} &
\hspace{-0.7cm}\includegraphics[width=4.7cm]{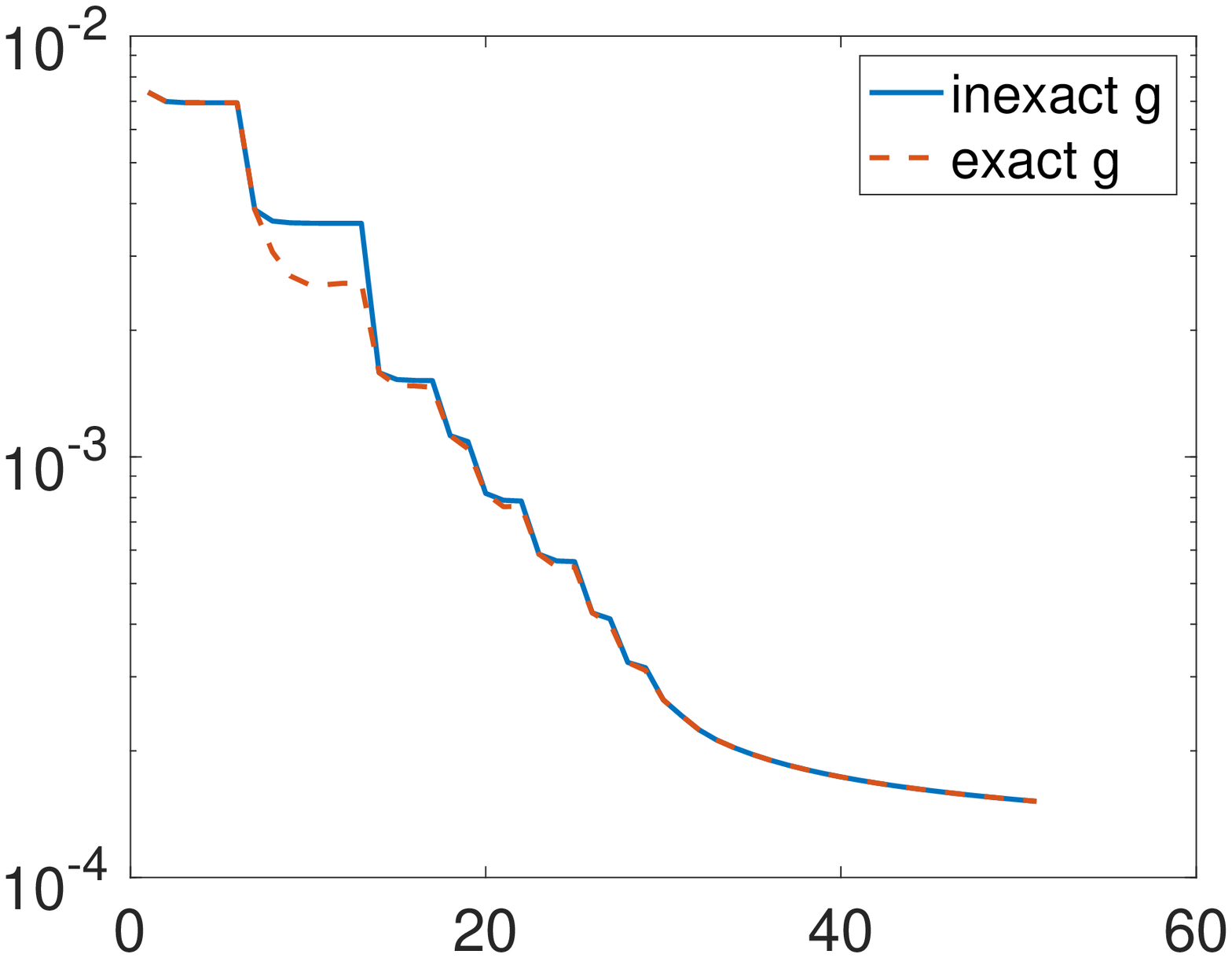} & 
\hspace{-0.7cm}\includegraphics[width=4.7cm]{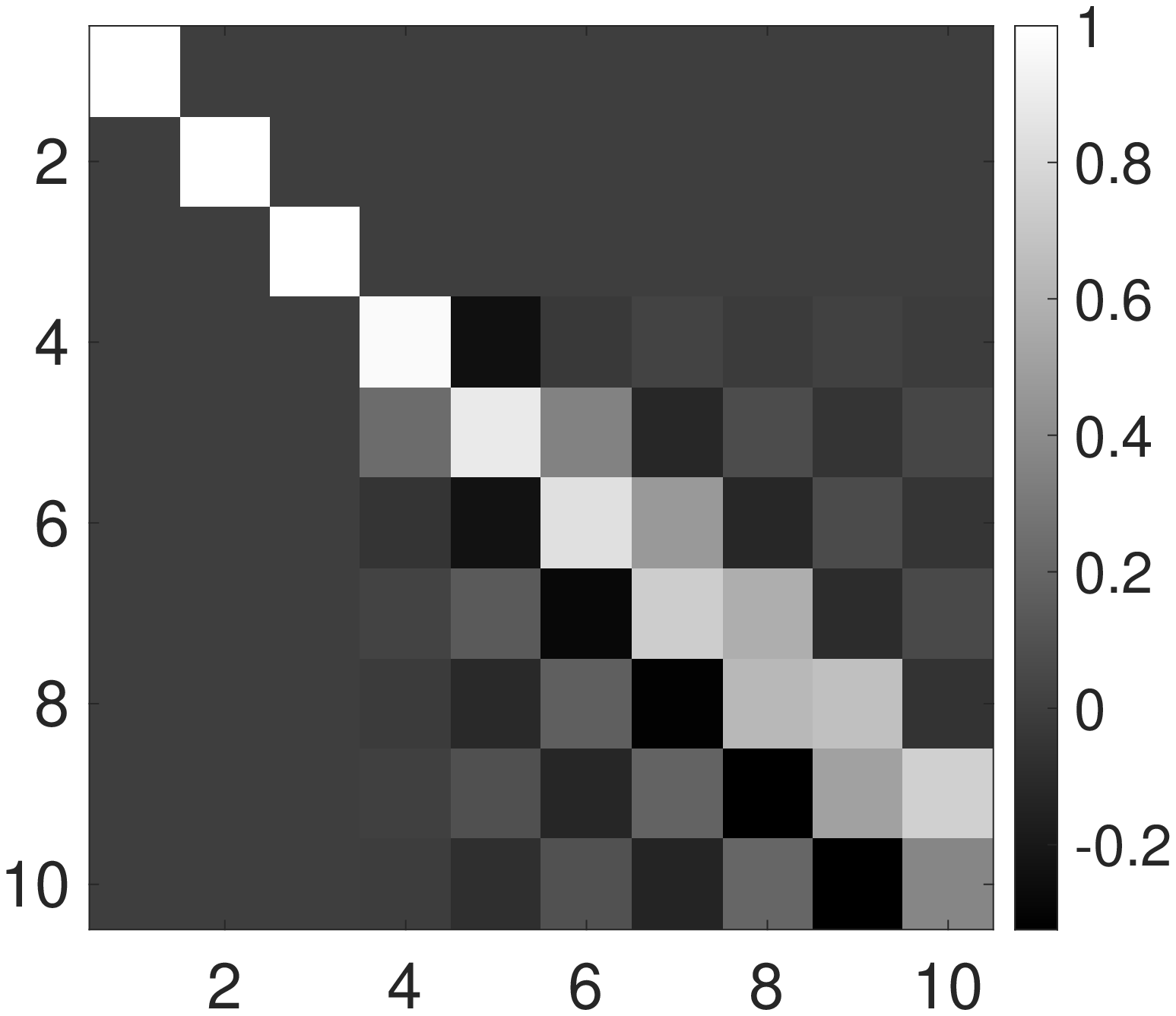}
\end{tabular}
\caption{Illustrative \texttt{satellite} test problem, for Algorithm \ref{alg:new} with $\lambda=5\cdot 10^{-1}$. \textbf{(a)} Relative errors versus total iterations. \textbf{(b)} Value of the blurring parameter $\sigma_1=\sigma_2$ versus total iterations. \textbf{(d)} Number of inner iterations versus number of outer iterations. \textbf{(e)} Exact and inexact objective function $g$ (computed using $\wtAlam(y_j)$ and $(\wtAlam(y_{j})+\wtmcE_{j+1}^{j+1})$, respectively), versus total iterations. The remaining frames display the entries of the matrix $(V_{10}^{\lambda_1})^TV_{10}^{\lambda_2}$, taking $\lambda_1=0$ and $\lambda_2=0.5$ (frame \textbf{(c)}) and $\lambda_1=0$ and $\lambda_2$ variable at each iteration (frame \textbf{(f)}): since $(V_{10}^{\lambda_1})^TV_{10}^{\lambda_2}\neq I$ the inexact Krylov subspace $\mcR(V_{10})$ is not shift-invariant.}\label{fig:theoryill}
\vspace{-0.5cm}
\end{figure}

\begin{figure}
\begin{tabular}{ccc}
\hspace{-1.0cm}{\small \bf (a) $\RRE_x$} & \hspace{-0.7cm}{\small \bf (b) $\RRE_x$}& \hspace{-0.7cm}{\small \bf (c) $\RRE_x$}\vspace{-0.4cm}\\\\
\hspace{-1.0cm}\includegraphics[width=4.7cm]{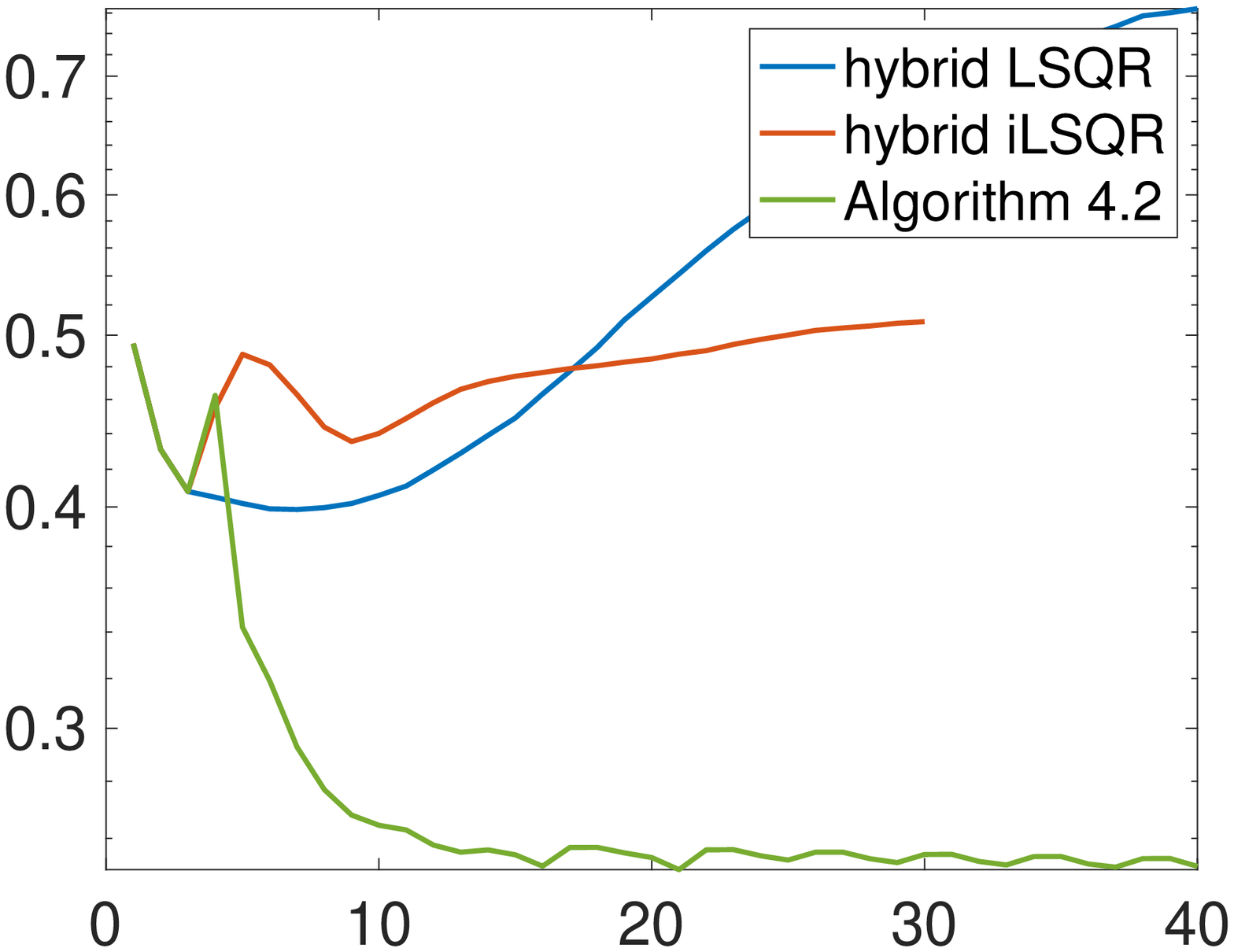} &
\hspace{-0.7cm}\includegraphics[width=4.7cm]{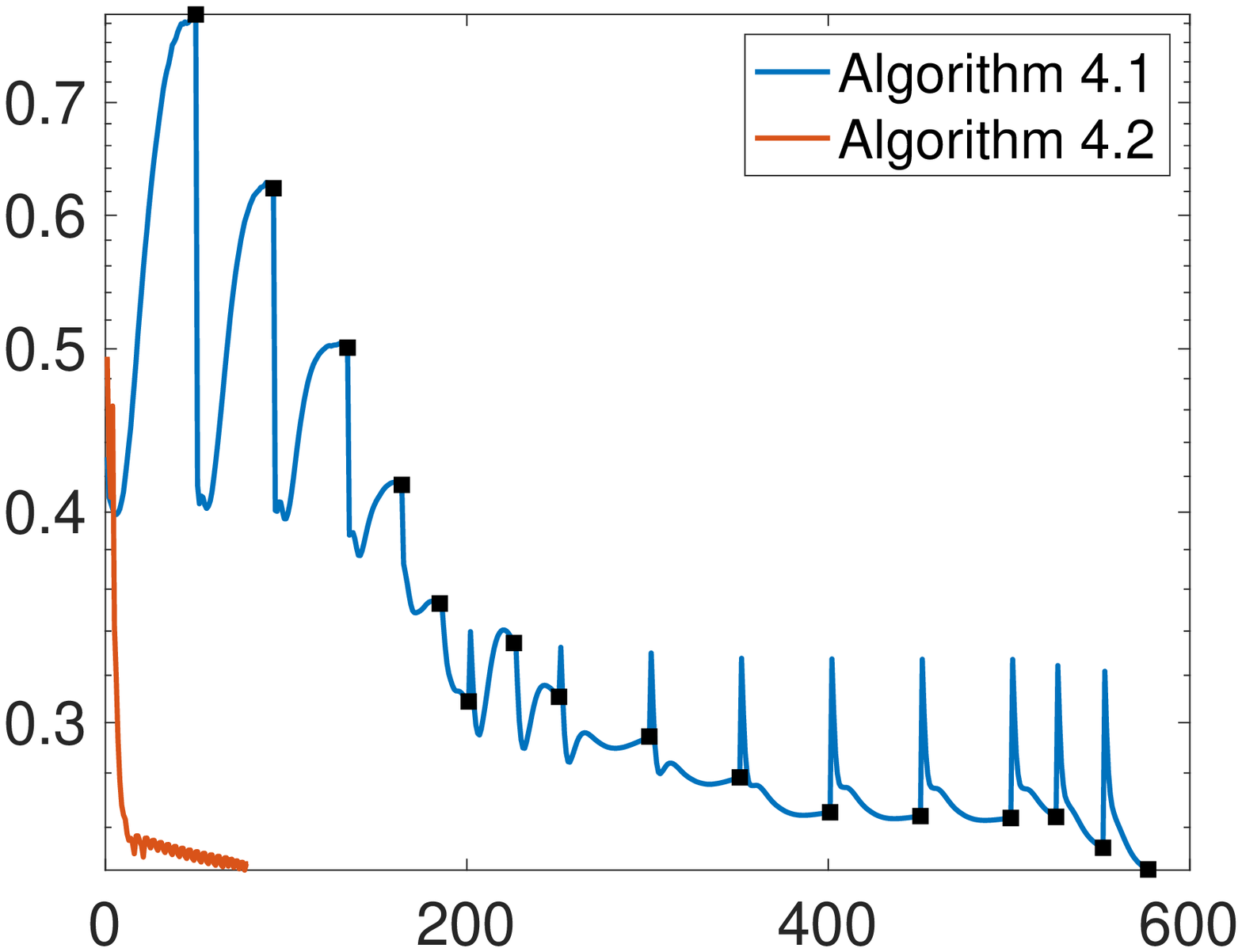} & 
\hspace{-0.7cm}\includegraphics[width=4.7cm]{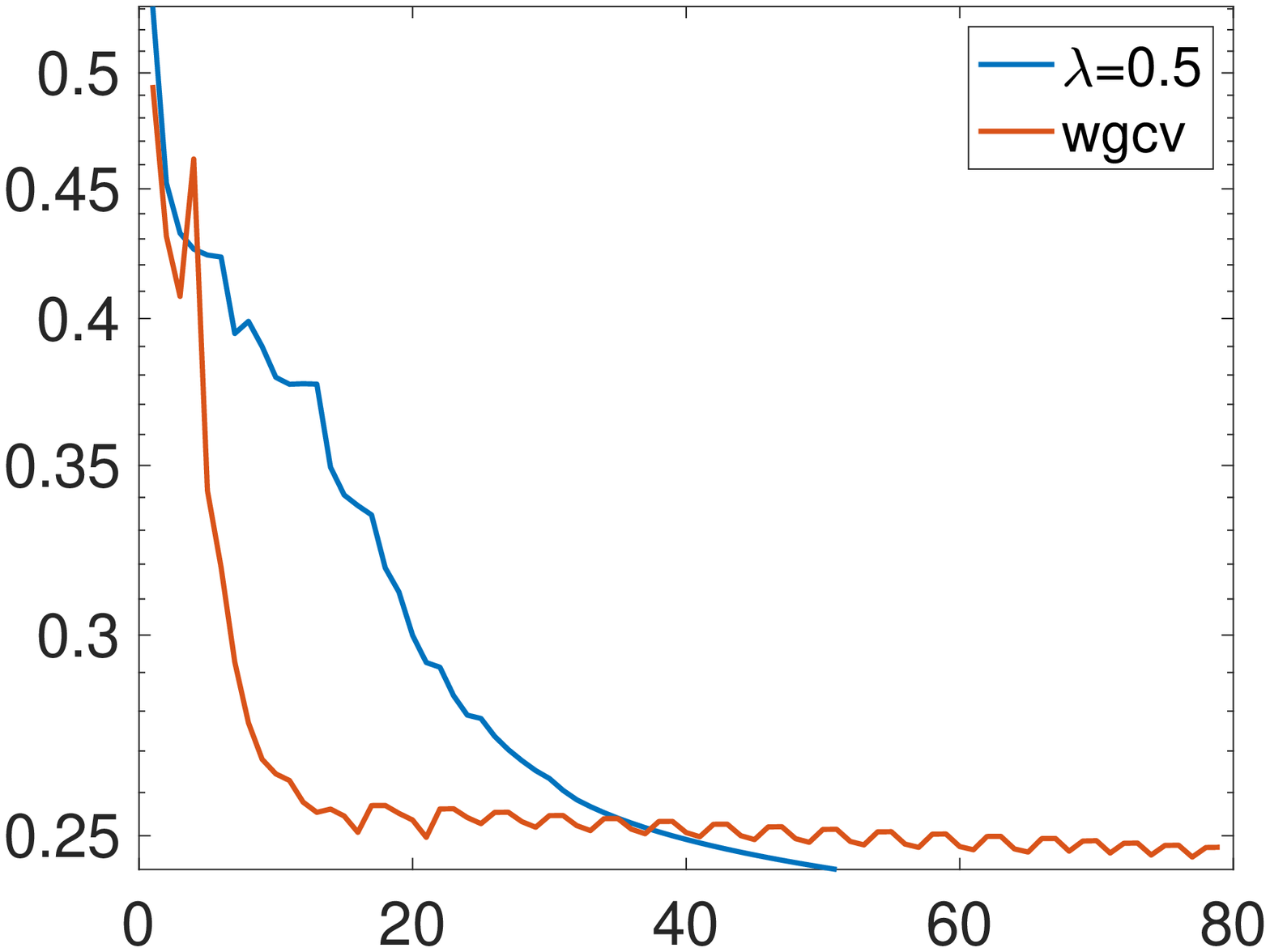}\vspace{-0.1cm}\\
\hspace{-1.0cm}{\small \bf (d) blur param. $y$} & \hspace{-0.7cm}{\small \bf (e) blur param. $y$} & \hspace{-0.7cm}{\small \bf (f) blur param. $y$}\vspace{-0.0cm}\\
\hspace{-1.0cm}\includegraphics[width=4.7cm]{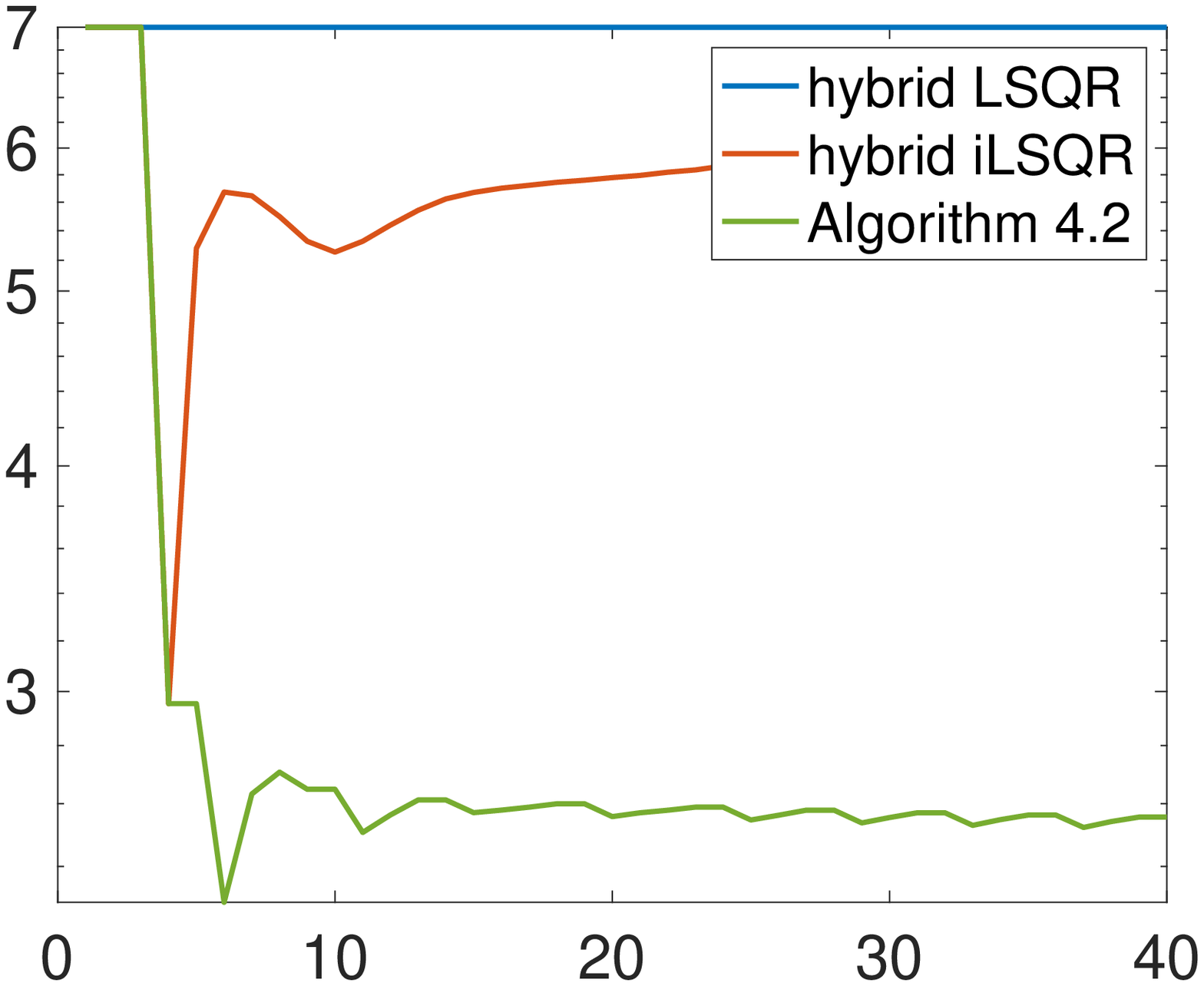} &
\hspace{-0.7cm}\includegraphics[width=4.7cm]{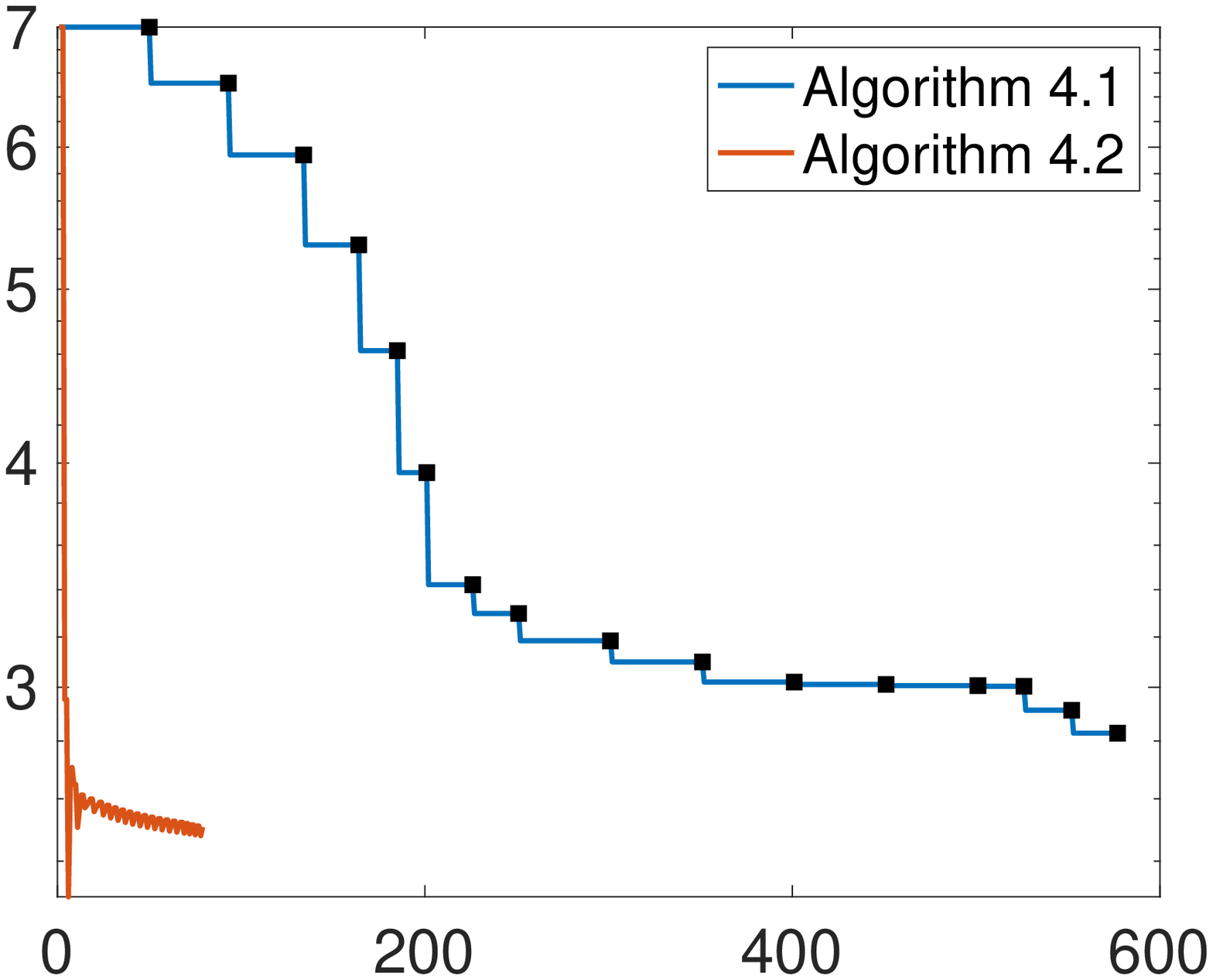} &
\hspace{-0.7cm}\includegraphics[width=4.7cm]{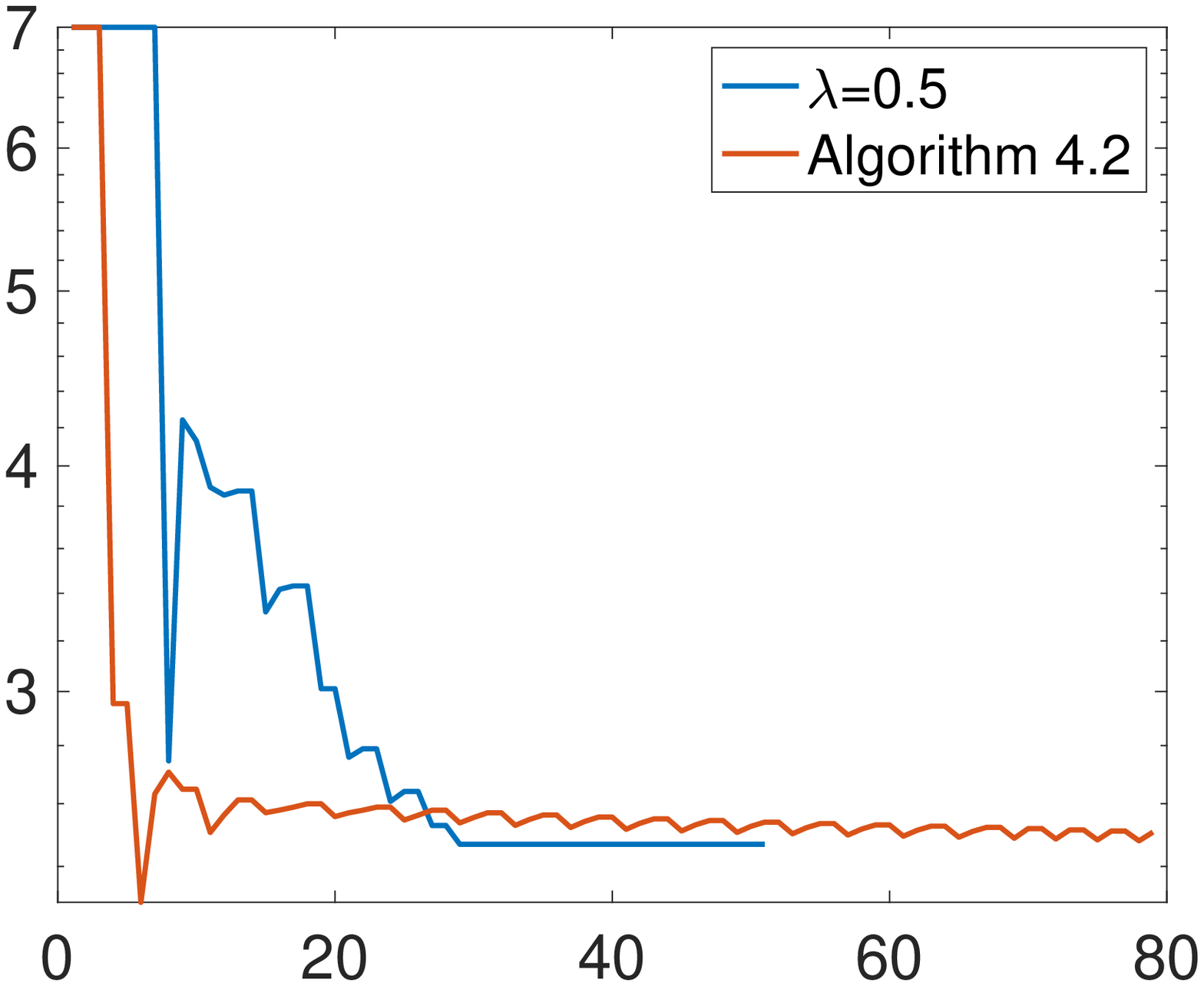}
\end{tabular}
\caption{Illustrative \texttt{satellite} test problem. {\bf (a)} Relative errors versus total iterations for the exact LSQR-based hybrid method (with $A=A(y_0)$), iLSQR-based hybrid method without error control, and Algorithm \ref{alg:new}. {\bf (b)} Relative errors versus total iterations for Algorithm \ref{alg:old} and Algorithm \ref{alg:new}. {\bf (c)} Relative errors versus total iterations for Algorithm \ref{alg:new} with fixed $\lambda$ and with iteration-dependent $\lambda$ chosen by wGCV.  {\bf (d)} Blurring parameter $\sigma_1=\sigma_2$ versus total iterations for LSQR-based hybrid method (with $y=y_0=7$), iLSQR-based hybrid method without error control, and Algorithm \ref{alg:new}. {\bf (e)} Blurring parameter $\sigma_1=\sigma_2$ versus total iterations for Algorithm \ref{alg:old} and Algorithm \ref{alg:new}. {\bf (f)} Blurring parameter $\sigma_1=\sigma_2$ versus total iterations for Algorithm \ref{alg:new} with fixed $\lambda$ and with iteration-dependent $\lambda$ chosen by wGCV. Black markers in {\bf (b)} and {\bf (e)} highlight the values at each outer iteration.}\label{fig:testing}
\vspace{-0.5cm}
\end{figure}

\begin{figure}
\begin{tabular}{ccc}
\hspace{-1.6cm}{\small \bf Hybrid-iLSQR} & \hspace{-1.6cm}{\small \bf Algorithm \ref{alg:old}} & \hspace{-1.6cm}{\small \bf Algorithm \ref{alg:new}}\vspace{-0.1cm}\\
\hspace{-1.6cm}{\small (it. 30, $\RRE_x$ 0.5819)} & \hspace{-1.6cm}{\small (it. 577, $\RRE_x$ 0.2454)} & \hspace{-1.6cm}{\small (it. 79, $\RRE_x$ 0.2474)}\vspace{-0.1cm}\\
\hspace{-1.6cm}\includegraphics[width=5.5cm]{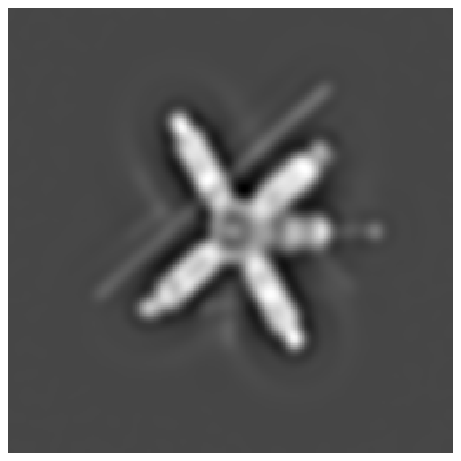} &  
\hspace{-1.6cm}\includegraphics[width=5.5cm]{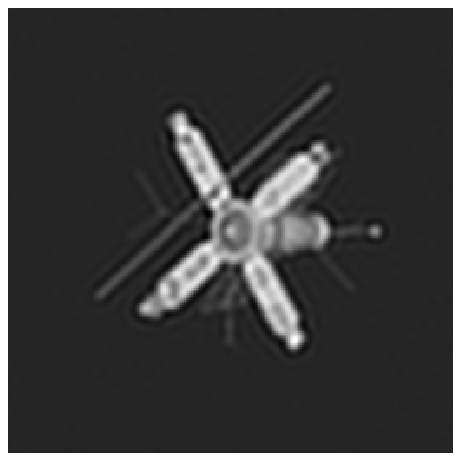} & 
\hspace{-1.6cm}\includegraphics[width=5.5cm]{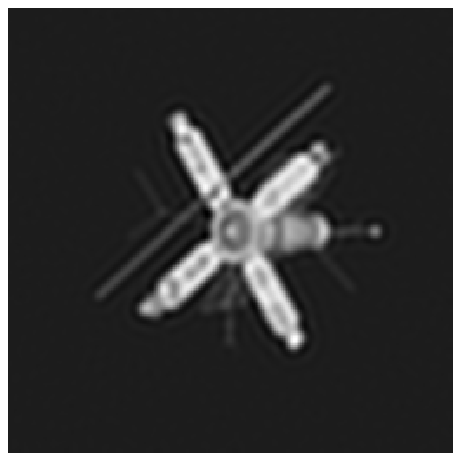}\vspace{-1.0cm}\\
\end{tabular}
\caption{Illustrative \texttt{satellite} test problem: images computed by different solvers (total iteration number and relative error are reported in brackets). Hybrid-iLSQR is implemented without error control. 
}\label{fig:reconstr}
\vspace{-0.5cm}
\end{figure}


\subsection{Computational strategies} 
In this section we discuss some implementation details, which were just briefly mentioned in the previous sections. In particular, we focus on the computation of error bounds for controlling inexactness of the solvers, adaptive regularization parameter choice, and stopping criteria. 
\paragraph{Computable bounds for $\|E_j\|$}
As mentioned in the previous sections and illustrated in frame (a) of Figure \ref{fig:testing}, error control must be implemented to expect meaningful approximations when using inexact methods.
In this section we describe how bounds for the errors can be cheaply obtained when considering blurring matrices. 

Under specific assumptions on the PSF and the boundary conditions, the coefficient matrix in (\ref{eq:linsys}) has a special structure that can be exploited for computing its eigendecomposition or its SVD: we refer to \cite[Chapter 4]{book1} for an overview. We remark that, within the blind deconvolution framework described in Section \ref{ssec:pb}, such assumptions (e.g., rank-1 or symmetric PSF, circulant boundary conditions) cannot generally be made, as they will result in corrupted reconstructions; see, for instance, \cite{DykesL2021} and the references therein. In the following, and in agreement with Section \ref{ssec:pb}, we provide details for the reflexive boundary conditions case. 
%

It is well known that, when the PSF is doubly symmetric and reflexive boundary conditions are imposed, the blurring matrix 
is a block Toeplitz-plus-Hankel matrix with Toeplitz-plus-Hankel blocks and can be diagonalized using the discrete cosine transform (DCT), which can be implemented as fast cosine transform (FCT); see \cite{Ng1999}. More precisely, at the $j$th iteration of Algorithm \ref{alg:new}, one can write
\[
A(y_{j-1})=C^T \Lambda_j C=C^T \text{diag}( \lambda^{(j)}_1,...,\lambda^{(j)}_n) C\,,\quad\mbox{where}\quad
\lambda^{(j)}_l = \,[CA(y_{j-1})e_1]_l/[C]_{l,1}
\]
and $C$ is the two-dimensional orthogonal discrete cosine transform matrix. Therefore, according to the definition in (\ref{def:error}), for $i\leq j$, 
\begin{eqnarray}
\|E_i^j\|=\|A(y_{i-1})-A(y_{j-1})\|=\|C^T (\Lambda_i-\Lambda_j) C\|&=&\|\Lambda_i-\Lambda_j \|\label{errdct}\\
&=&\max_{k=1,\dots,n}\left|\lambda^{(i)}_k-\lambda^{(j)}_k\right\vert\,.\nonumber
\end{eqnarray}
Moreover, thanks to the normalization condition for the PSF in (\ref{GaussPSF}), it follows that, in this case, $\|A(y_{i-1})\|=1$, $i\leq j$: this is useful if the bound (\ref{hiLSQRbound}), involving the estimate (\ref{zbound}) is employed. 

If the PSF is not doubly symmetric, then one can replace the blurring matrix with its optimal (in the Frobenius norm $\|\cdot\|_F$) approximation obtained through cosine transformation. Namely, at the $j$th iteration of Algorithm \ref{alg:new}, one approximates $A(y_{j-1})$ by 
\begin{equation}\label{dctapprox}
\bar{A}(y_{j-1})=C^T\bar{\Lambda}_jC,\quad\mbox{where}\quad  \bar{\Lambda}_j=\mbox{diag}\left(\frac{1}{2}C(A(y_{j-1}) + A^T(y_{j-1}))C^T\right)\,.
\end{equation}
In other words, one should just take the symmetric part of the PSF and form a structured matrix with respect to it. The error associated to such optimal approximation is $\|\bar{A}(y_{j-1})-A(y_{j-1})\|_F^2=1/2\|C(A(y_{j-1})-A^T(y_{j-1}))C^T\|_F^2$, i.e., $\bar{A}(y_{j-1})$ is a good approximation of $A(y_{j-1})$ if the blurring matrix (or the PSF $P(y_{j-1})$) is close to symmetric. The approximation (\ref{dctapprox}) is typically used when devising preconditioners for image deblurring problems (\ref{eq:linsys}); see again \cite{Ng1999}. Within framework described in Section \ref{ssec:solerror} we propose to use approximation (\ref{dctapprox}) to efficiently control the inexactness of the hybird-iLSQR methods, i.e., for $i\leq j$, we take
\begin{equation}\label{errdctapprox}
\|{E}_i^j\|\simeq \|\bar{E}_i^j\|:=\|\bar{A}(y_{i-1})-\bar{A}(y_{j-1})\|=\|\bar{\Lambda}_i-\bar{\Lambda}_j \|\,.
\end{equation}

\paragraph{Parameter choice}
As already remarked in \cite{JulianneJim}, 
being able to adaptively set the regularization parameter $\lambda$ in (\ref{eqtikhybis}) is of pivotal importance when performing blind deconvolution, since the amount of regularization should be linked to the amount of blur, as specified by the varying unknown blur parameters $y$.
When using inexact hybrid methods in this setting, since $y$ depends on the current $\lambda$ through the current approximate solution $x$, it cannot generally be guaranteed that the assumptions of Proposition \ref{prop:spacei} hold, i.e., the approximation subspace for the solution $x$ may not be invariant with respect to $\lambda$. For instance, for the illustrative example described in Section \ref{ssec:pb}, all the possible cross products of the orthonormal basis vectors obtained after 10 hybrid-iLSQR iterations performed with different values of $\lambda$ are nonzero, implying that such basis vectors span two different spaces. These quantities are displayed in frames (c) and (f) of Figure \ref{fig:theoryill}.  
As a consequence, although one can successfully regularize the projected inexact problem (as described below), this may not be equivalent to regularizing the exact full-dimensional problem.

If a good estimate of the magnitude of the noise $\|e\|$ is available, we can apply the discrepancy principle to the projected problem (\ref{eq:iprojregaugpb_1}), i.e., at the $k$th hybrid-iLSQR iteration we compute $\lambda=\lambda_k$ such that
\[
\|M_ks_{\lambda,k}-\beta e_1\|^2 = \tau\|e\|^2\quad\mbox{where $\tau$ is a safety threshold (typically $\tau>1$, $\tau\simeq 1$)}.
\]
We note that satisfying the above condition does not guarantee that the `exact' discrepancy principle is satisfied. Indeed, similarly to (\ref{iLSQRbound}), we get
\begin{eqnarray*}
\tau\|e\|^2 - \eps\leq \|Ax_{\lambda,k}- b\|^2
\leq \tau\|e\|^2 + \eps\,,\quad\mbox{where}\quad
\eps=\|E_0x_0\|^2 + \sum_{l=1}^k\|E_l\|^2|[s\lamk]_l |^2
\end{eqnarray*}
is controlled through the inexactness bounds derived in Section \ref{sec: theory}. Recalling that $\|Ax_{\lambda,k}- b\|$ is an increasing function of $\lambda$, and depending on the above bounds being quite strict or loose, $x_{\lambda,k}$ may be under- or over- regularized for the original problem.

Alternatively, following \cite{JulianneJim}, we may use the weighted GCV (wGCV) criterion applied to the projected problem (\ref{eq:iprojregaugpb_1}), i.e., at the $k$th hybrid-iLSQR iteration we compute
\[
\lambda_k=\arg\min_{\lambda\geq 0}\frac{k\|(I-M_k(M_k^TM_k+\lambda^2I)^{-1}M_k^T)\beta e_1\|^2}{(\mbox{trace}(I-\omega M_k(M_k^TM_k+\lambda^2I)^{-1}M_k^T))^2}\,.
\]
We remark that other parameter choice rules typically employed within (exact) hybrid methods (see, for instance, \cite{survey, newsurvey}), can be adapted to work with inexactness; 
moreover, other strategies that rely on structured approximations of the blurring matrix can be exploited, too (as suggested by \cite{chungframe}). 

\paragraph{Stopping criteria}
When solving the blind deconvolution problem, one is recovering the unknown sharp image $x\in\R^n$ as well as the unknown blurring parameters $y\in\R^p$: for this reason, effective stopping criteria should be devised, based on the behavior of both variables. Moreover, when inexact methods are applied as described in Section \ref{ssec:solerror}, the effect of the errors in the estimated blurring matrix has to be considered: indeed, as specified in line 3 of Algorithm \ref{alg:new}, one should restart the hybrid-iLSQR method as soon as the difference between exact and inexact residual (bounded as in (\ref{hiLSQRbound}) or (\ref{iLSQRboundalt})) exceeds a pre-specified or adaptively estimated tolerance $\eps$. 
%
%
Looking at the progress of both $x$ and $y$ it would be natural to stop the iterations of Algorithms \ref{alg:new} as soon as the relative gradient norm $\|\nabla_yh(y)\|/\|\nabla_yh(y_0)\|$ of the objective function is approximately zero, which means that a stationary point for the objective function $h$ (defined in (\ref{hdef})) has been reached. 
In addition to this, one may monitor the (relative) stabilization of some relevant quantities, e.g., stop when 
\[
\frac{|\lambda_k-\lambda_{k-1}|}{\lambda_{k-1}}\leq \theta_1\,,\quad
\frac{\|x\lamk-x_{\lambda,k-1}\|}{\|x_{\lambda,k-1}\|}\leq \theta_2\,,\quad
\frac{\|y_k-y_{k-1}\|}{\|y_{k-1}\|}\leq\theta_3\,,\quad k=2,3, \dots,
\]
where $\theta_1,\,\theta_2,\,\theta_3>0$ are user-specified tolerances. If wGCV is employed to set $\lambda$, it can be also (simultaneously) used as a stopping criterion; we refer to \cite{JulianneJim} for additional details.

\paragraph{Further comments on 
hybrid-iLSQR for the illustrative test problem 
in Figure \ref{fig:testP}}
We conclude this section by providing some comments about the performance of the hybrid-iLSQR method, especially in comparison with the well-established Algorithm \ref{alg:old}, which is implemented with `cold' restarts (i.e., taking $x_0=0$ at line 10); 
Algorithm \ref{alg:new} is instead implemented with `warm' restarts, i.e., taking $x_0=x\lamk$ at line 10. Note that, to enforce that $\sigma_1=\sigma_2$ (and $\rho=0$) in (\ref{GaussPSF}) and keep the illustrative example simple, the solvers are coded in such a way that only one blurring parameter, i.e., $y=\sigma_1$, has to be computed. Looking at Figures \ref{fig:testing} and \ref{fig:reconstr}, it is evident that both Algorithms \ref{alg:old} and \ref{alg:new} eventually compute reconstructions of the same quality, as the values of the relative errors and the blurring parameter are quite similar; in particular, since $\sigma_1=2.5$, $y$ is better approximated using Algorithm \ref{alg:new}). All the graphs in Figure \ref{fig:testing} display the behavior of the methods versus the total number of iterations. In particular, the first inner loops of Algorithm \ref{alg:old} are affected by the so-called semi-convergence phenomenon (i.e., permanent increase of the error after only a few iterations): this is evident looking at frame (b), and can probably be mitigated by a more accurate tuning of the inner stopping criteria; nevertheless, the values at the outer iterations are generally decreasing. Some oscillations in the reconstruction quality also affect Algorithm \ref{alg:new} implemented with the wGCV criterion, while we note that, for this test problem, the behavior versus the number of iterations seems more stable when a fixed regularization parameter is employed (see Figure \ref{fig:testP}, frames (c) and (f)). Algorithm \ref{alg:new} is implemented with error control as described in (\ref{errdct}): note that the FCT-based decomposition appearing therein is exact for this test problem, as $\sigma_1=\sigma_2$ is enforced at each iteration, implying that the PSF is doubly symmetric. Finally, we remark that the performance of both Algorithms \ref{alg:old} and \ref{alg:new} depends on the initial guess for the blurring parameters (for the examples shown here, $y_0=7$).

The cost of $k$ iterations of Algorithms \ref{alg:old} and \ref{alg:new} is comparable when $k\ll n$. Indeed, both algorithms have to compute $k$ matrix-vector products with $A$ and $A^T$ to generate the approximation subspace for $x$: while this in general costs $O(kn^2)$ flops, it can be reduced to $O(kn\log n)$ if additional assumptions on the PSF (and the boundary conditions) hold; see \cite[Chapter 4]{book1} for an overview. Generating an orthonormal basis for the solution subspace amounts to $O(n)$ flops for Algorithm \ref{alg:old} (thanks to short recurrences) and $O(k^2n)$ flops for Algorithm \ref{alg:new} (because of full orthonormalization). Solving the projected problem costs $O(k)$ flops for Algorithm \ref{alg:old} (exploiting the bidiagonal structure of the matrix in (\ref{eq:projregpb})) and $O(k^2)$ for Algorithm \ref{alg:new}. Following the startegy in \cite{JulianneJim}, the cost of updating the blurring parameters for both algorithms amounts to $O(pn)$ flops for computing the Jacobian (\ref{def:Jacob}) and $O(p^2)$ for performing a Gauss-Newton step (\ref{yupdateapprox}): these are negligible if $p\ll n$.

\subsection{Numerical experiments}\label{ssec:numer}

To further validate the performance of Algorithm \ref{alg:new} we display the results of one additional blind image deblurring test problem: we take the \texttt{cameraman} test image of size $256\times 256$ pixels (shown in the top left frame of Figure \ref{fig:cameraim}) and we corrupt it by applying a Gaussian blur (\ref{GaussPSF}) with parameters $y_\true=[3, 4, 0.5]^T$ and Gaussian white noise of level $\|e\|/\|b_\true\|=10^{-2}$. We start both Algorithms \ref{alg:old} and \ref{alg:new} with initial guesses $x_0=0$ and $y_0=[5, 6, 1]^T$. Figure \ref{fig:cameraim} displays the reconstruction of the images and the PSFs obtained by the two methods: although the values of $\RRE_x$ are comparable, the image computed by Algorithm \ref{alg:old} still appears slightly blurred, while the image computed by Algorithm \ref{alg:new} appears sharper but currupted by some artefacts; Algorithm \ref{alg:old} computes a better approximated PSF than Algorithm \ref{alg:new}: indeed, as it can be also seen in frames (b) and (e) of Figure \ref{fig:cameracomp}, the parameter $\rho$ governing the orientation of the PSF is overestimated by Algorithm \ref{alg:new}. The remaining frames of Figure \ref{fig:cameracomp} display the behavior of relevant quantities computed by Algorithm \ref{alg:old} and different variants of Algorithm \ref{alg:new}, versus the (total) number of iterations. As it can be seen in frame (a), although the first inner loops of Algorithm \ref{alg:old} are affected by semi-convergence, the relative error decreases as the outer iterations proceed (especially during the final cycles). As it can be seen in frames (c) and (f), not considering error control in Algorithm \ref{alg:new} results in spoiled reconstructions; the discrepancy principle performs better than wGCV when used to adaptively set the regularization parameter for this test problem. 

\begin{figure}
\begin{tabular}{ccc}
\hspace{-1.6cm}{\small \bf exact} & \hspace{-1.6cm}{\small \bf Algorithm \ref{alg:old}} & \hspace{-1.6cm}{\small \bf Algorithm \ref{alg:new}}\\
\hspace{-1.6cm} & \hspace{-1.6cm}{\small (it. 927, $\RRE_x$ 0.1286)} & \hspace{-1.6cm}{\small (it. 82, $\RRE_x$ 0.1219)}\vspace{-0.0cm}\\
%
\hspace{-1.6cm}\includegraphics[width=5.5cm]{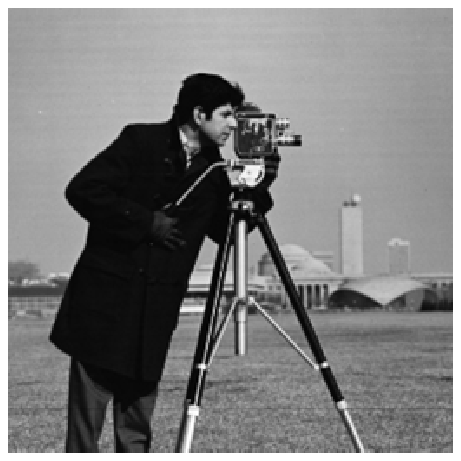} &
\hspace{-1.6cm}\includegraphics[width=5.5cm]{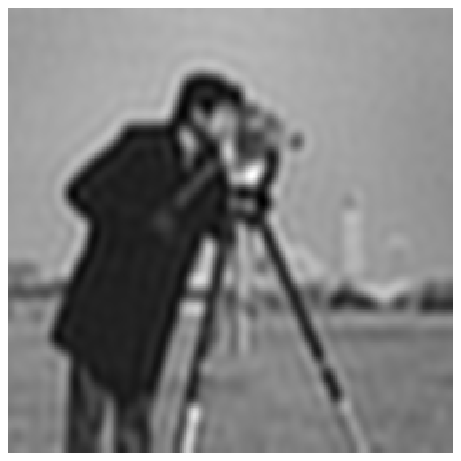} & 
\hspace{-1.6cm}\includegraphics[width=5.5cm]{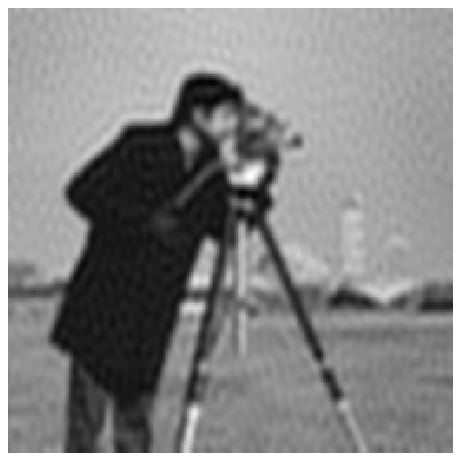}\vspace{-0.9cm}\\
\hspace{-1.6cm} & \hspace{-1.6cm}{\small (it. 927, $\RRE_y$ 0.0679)} & \hspace{-1.6cm}{\small (it. 82, $\RRE_y$ 0.1438)}\vspace{-0.0cm}\\
\hspace{-1.6cm}\includegraphics[width=5.5cm]{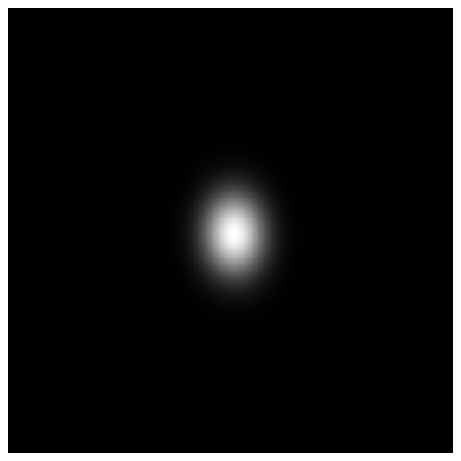} &
\hspace{-1.6cm}\includegraphics[width=5.5cm]{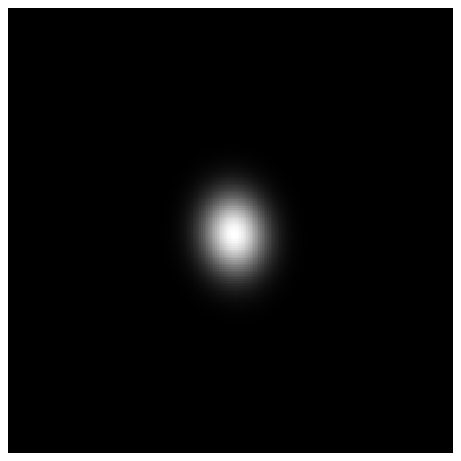} & 
\hspace{-1.6cm}\includegraphics[width=5.5cm]{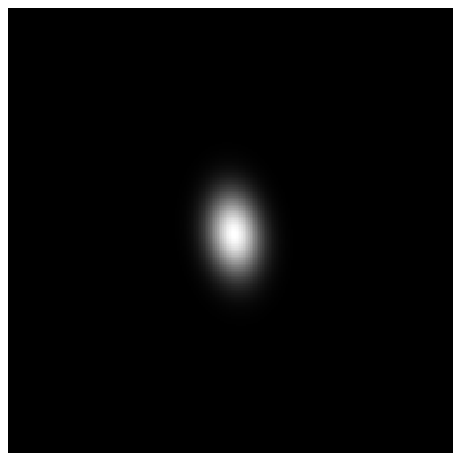}\vspace{-1.0cm}
\end{tabular}
\caption{\texttt{cameraman} test problem. Exact quantities and reconstructions thereof; 
total iteration number and relative reconstruction error are reported in brackets. 
}\label{fig:cameraim}\vspace{-0.5cm}
\end{figure}

\begin{figure}
\begin{tabular}{ccc}
\hspace{-1.0cm}{\small \bf (a) $\RRE_x$} & \hspace{-0.7cm}{\small \bf (b) blur param. $y$}& \hspace{-0.7cm}{\small \bf (c) $\RRE_x$}\vspace{-0.0cm}\\
\hspace{-1.0cm}\includegraphics[width=4.7cm]{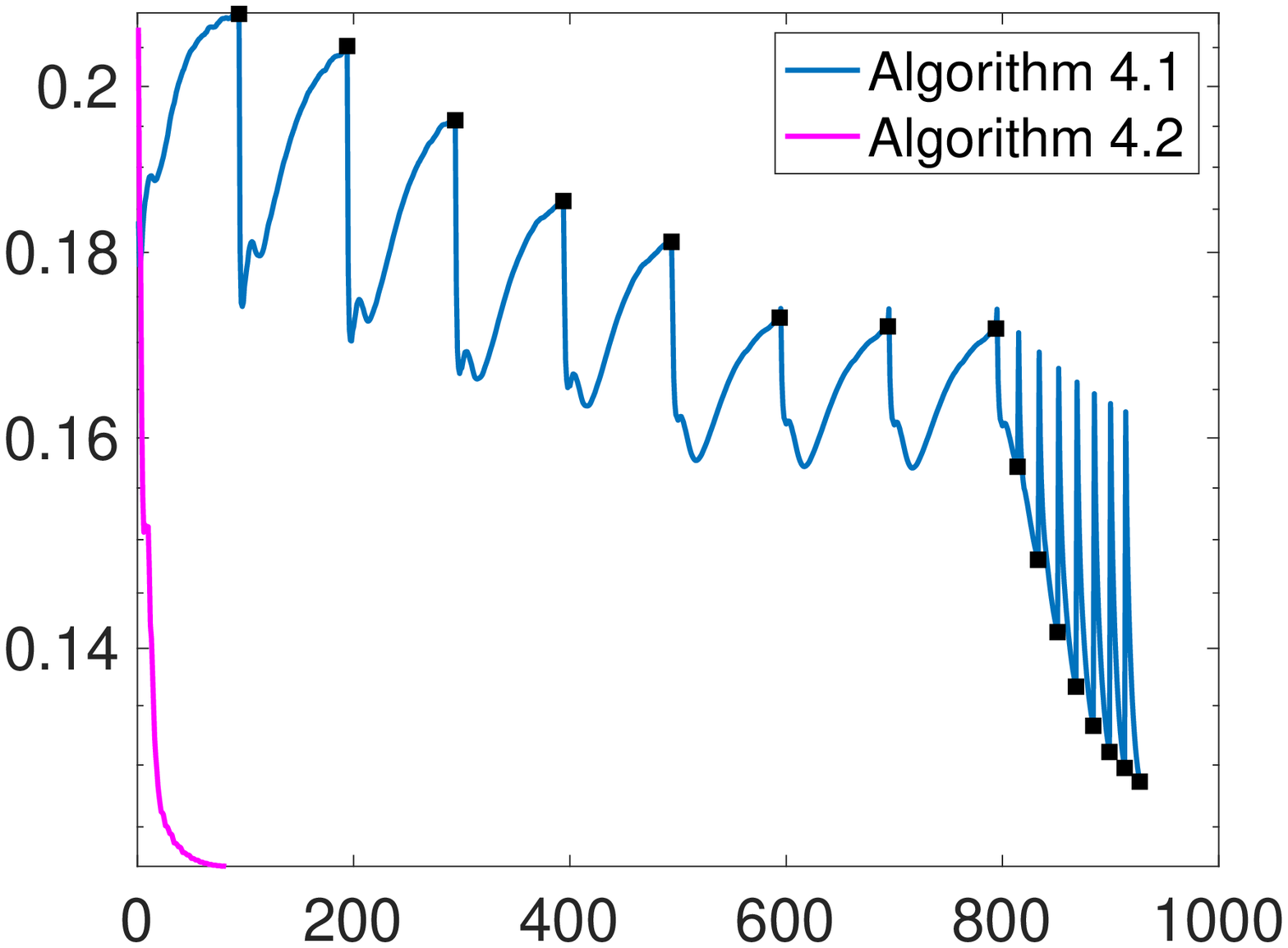} &
\hspace{-0.7cm}\includegraphics[width=4.7cm]{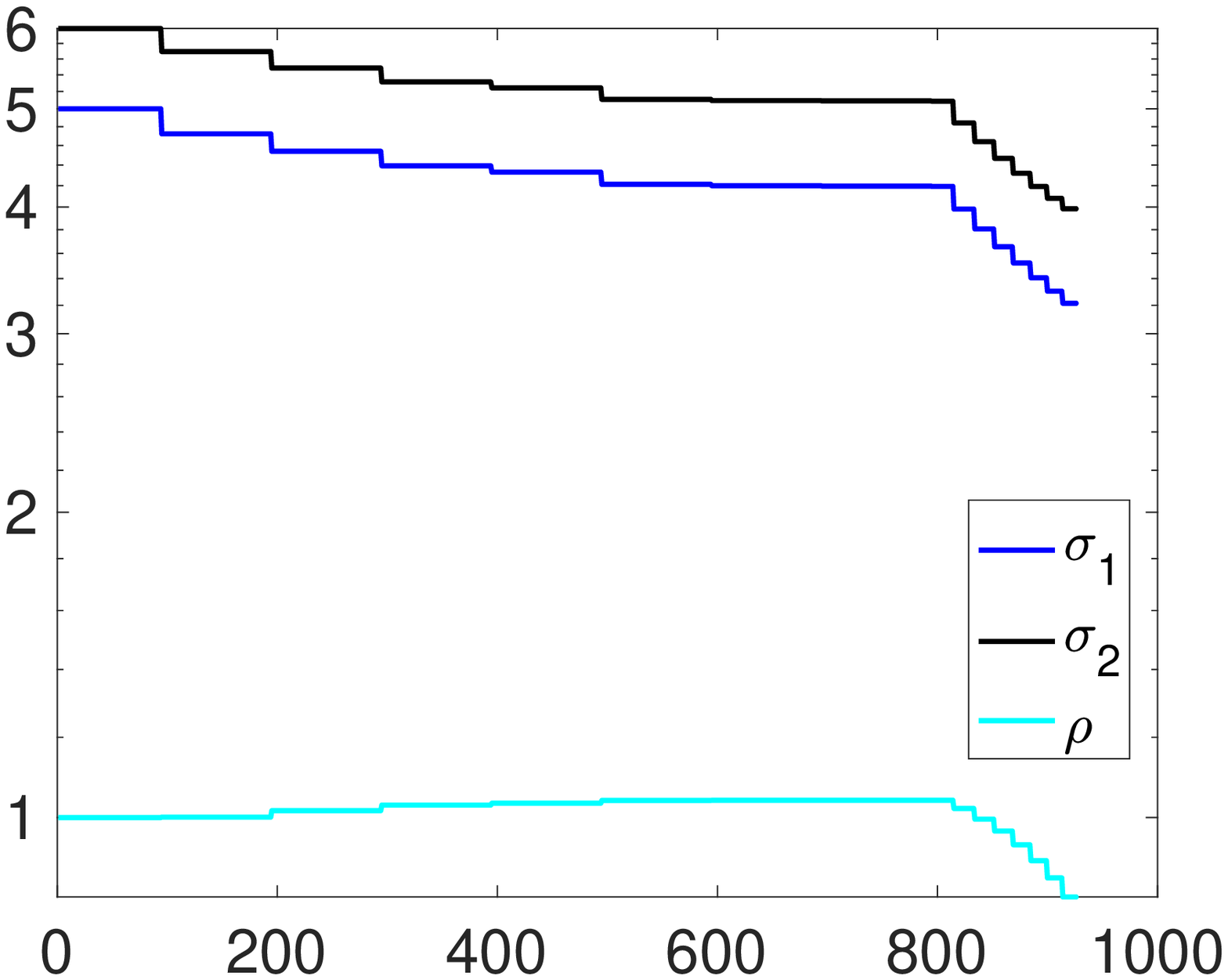} &
\hspace{-0.7cm}\includegraphics[width=4.7cm]{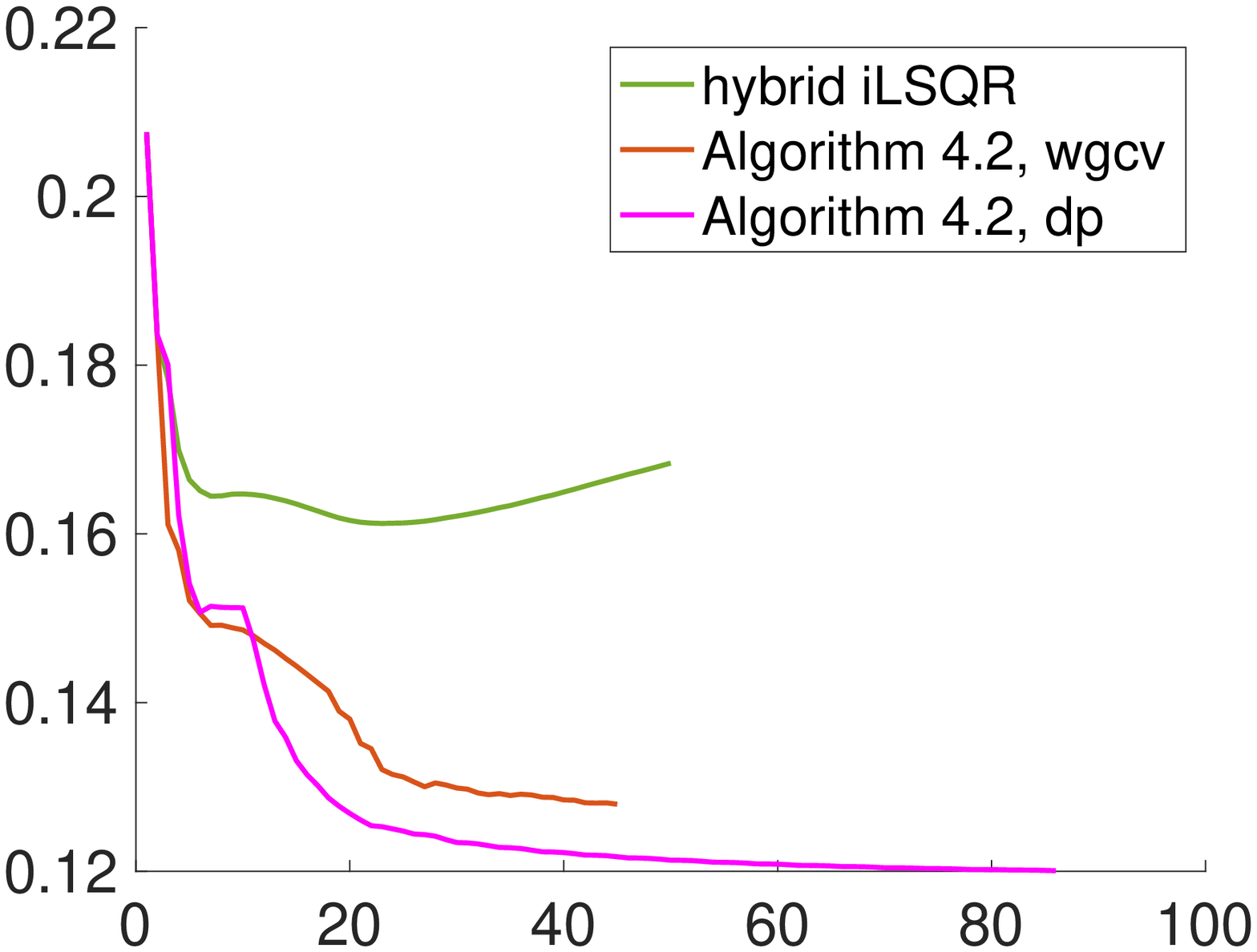}\\
\hspace{-1.0cm}{\small \bf (d) $\RRE_y$} & \hspace{-0.7cm}{\small \bf (e) blur param. $y$} & \hspace{-0.7cm}{\small \bf (f) $\RRE_y$}\vspace{-0.0cm}\\
\hspace{-1.0cm}\includegraphics[width=4.7cm]{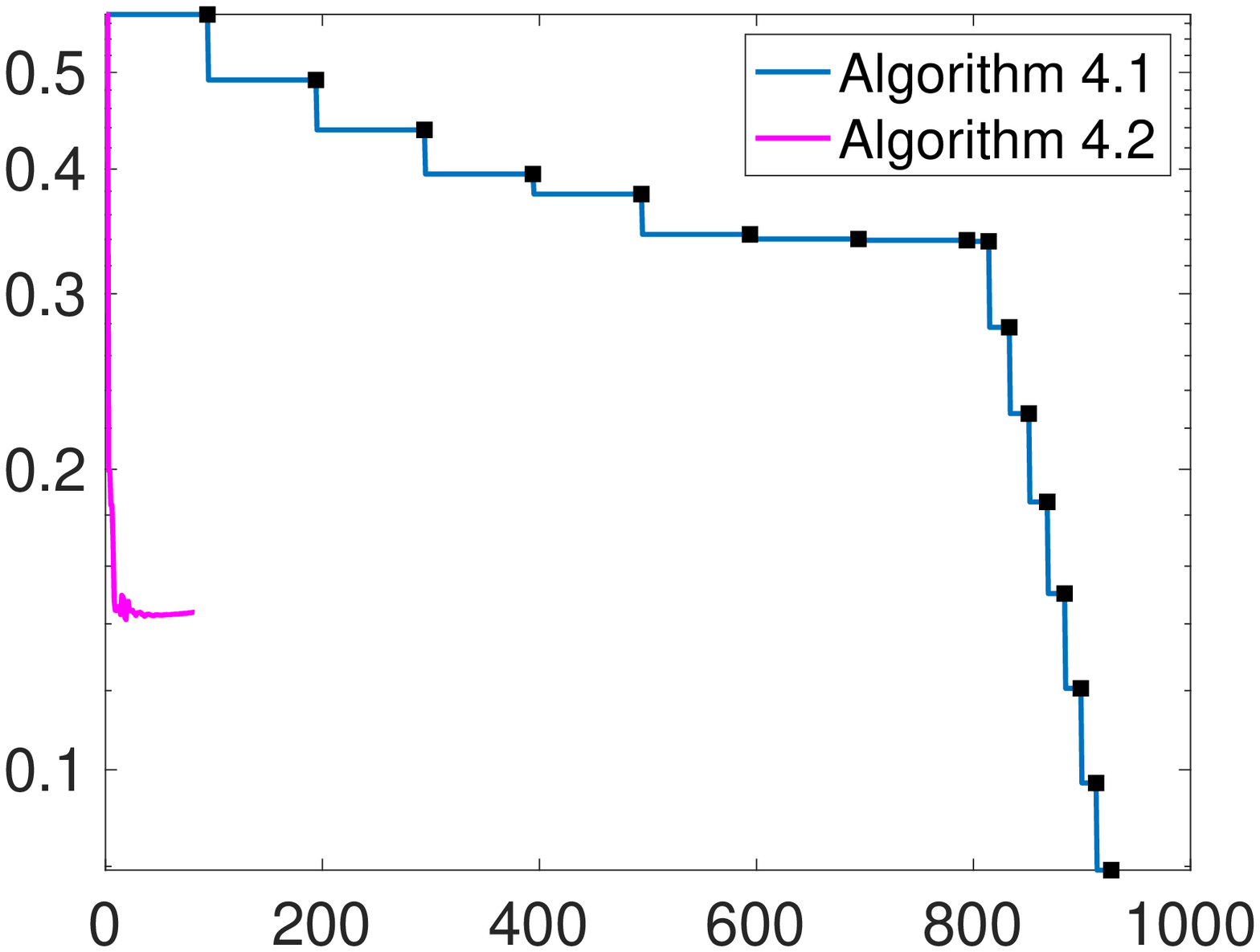} &
\hspace{-0.7cm}\includegraphics[width=4.7cm]{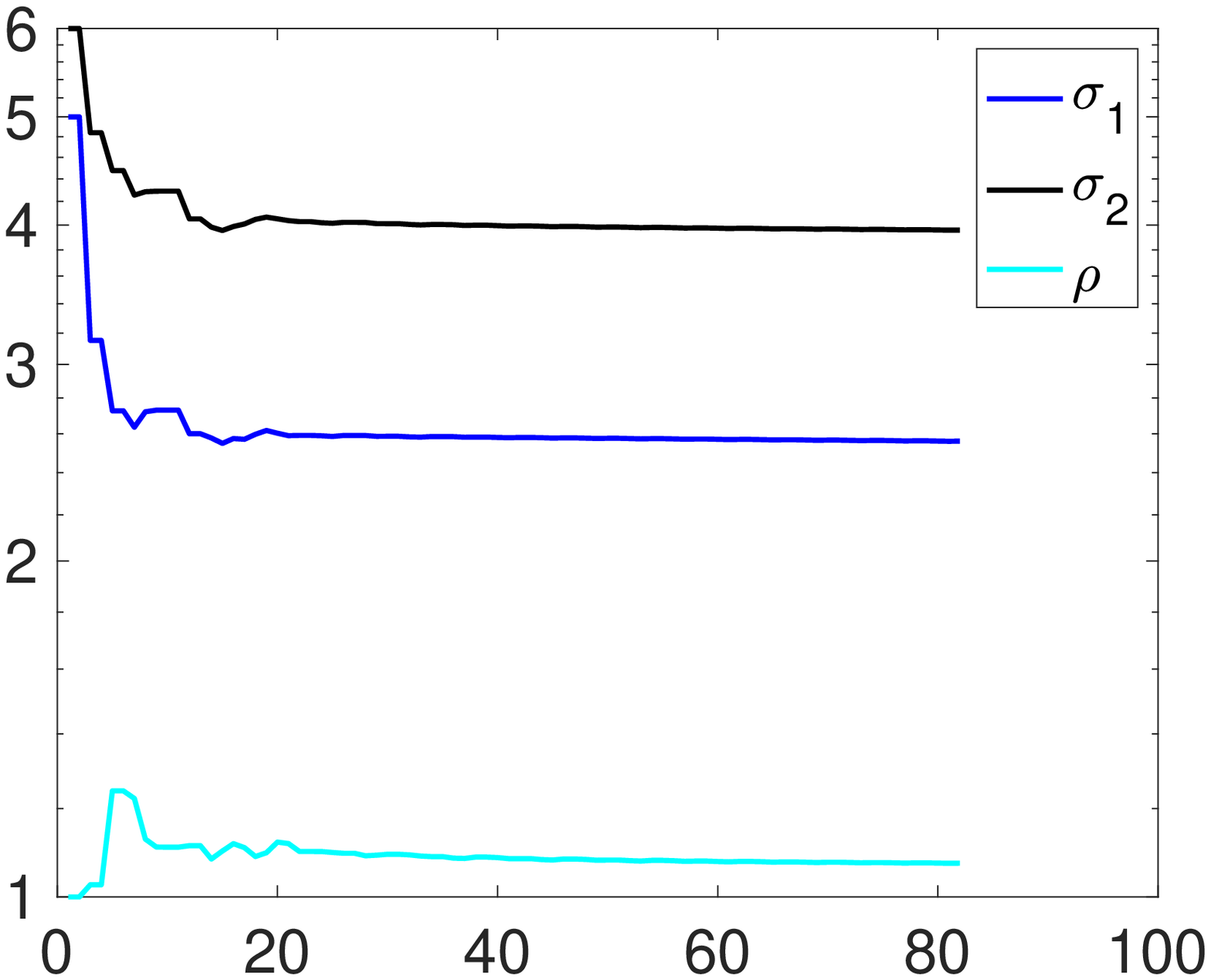} &
\hspace{-0.7cm}\includegraphics[width=4.7cm]{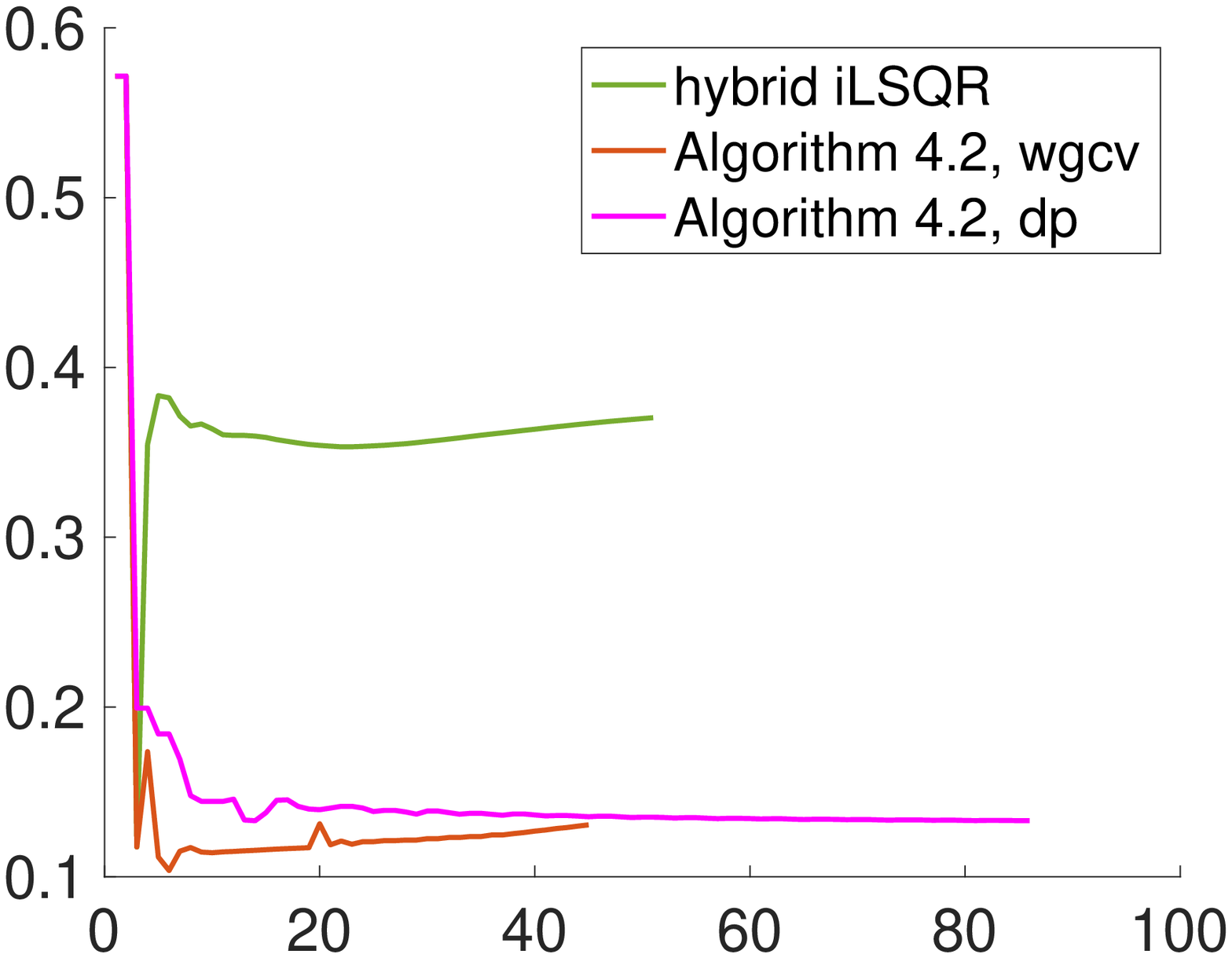}
\end{tabular}
\caption{\texttt{cameramen} test problem. {\bf (a)} $\RRE_x$ versus total iterations for Algorithms \ref{alg:old} and \ref{alg:new} implemented with the discrepancy principle. {\bf (b)} Values of the blurring parameter $y=[\sigma_1,\sigma_2,\rho]^T$ versus total iterations for Algorithm \ref{alg:old}. {\bf (c)} $\RRE_x$ versus iterations for the hybrid-iLSQR method without error control, and Algorithm \ref{alg:new} with the wGCV criterion or the discrepancy principle for setting $\lambda$.  {\bf (d)} $\RRE_y$ versus total iterations for Algorithms \ref{alg:old} and \ref{alg:new} implemented with the discrepancy principle. {\bf (e)} Values of the blurring parameter $y=[\sigma_1,\sigma_2,\rho]^T$ versus iterations for Algorithm \ref{alg:new}. {\bf (f)} $\RRE_y$ versus iterations for the hybrid-iLSQR method without error control, and Algorithm \ref{alg:new} with the wGCV criterion or the discrepancy principle for setting $\lambda$. Black markers in {\bf (a)} and {\bf (d)} highlight the values at each outer iteration.}\label{fig:cameracomp}\vspace{-0.5cm}
\end{figure}

\section{Conclusions and outlook}\label{sec:end}

In this paper we introduced and analysed the new iLSQR and iCGLS methods: these are inexact Krylov methods based on the inexact Golub-Kahan decomposition that, when used as purely iterative methods, or in combination with Tikhonov regularization, can be efficiently employed to regularize large-scale ill-posed problems, provided that the amount of inexactness is monitored at each iteration. We tested the new methods on separable nonlinear inverse problems arising in 
blind deblurring, handled with a variable projection approach. 

Future work will be performed with the goals of: (i) deriving new inexact solvers other than iLSQR and iCGLS, combined with methods other than standard form Tikhonov; (ii) handling nonlinear separable inverse problems other than blind deblurring. Regarding (i): still leveraging the inexact Golub-Kahan decomposition, one may devise an inexact version of LSMR, similarly to what was done for flexible solvers; see \cite{JulianneSilvia}. Alternatively, one may employ the inexact GMRES and FOM solvers based on the inexact Arnoldi decomposition described in \cite{inexact1}. Moreover, one may consider the combination of inexact methods and general-form Tikhonov regularization, where the penality term $\lambda^2\|Lz\|_2^2$ replaces $\lambda^2\|z\|_2^2$ in (\ref{tikhx0}), and where the regularization matrix $L\in\R^{n'\times n}$ enforces some prior information about the solution. 
Regarding (ii): the survey paper \cite{GolubVP} describes a number of applications where the variable projection method is routinely used and that, therefore, may be potentially adapted to work with inexact Krylov methods. These include many inverse problems, such as magnetic resonance imaging in medicine, superresolution of images, instrumental calibration, but also 
machine learning tasks such as the training of neural networks. 

\section{Acknowledgements}
We would like to thank Julianne Chung for sharing with us some of the codes used to produce the results in \cite{JulianneJim}.

\bibliographystyle{siamplain}
\bibliography{biblio}

\begin{thebibliography}{10}

\bibitem{CMRS00}
{\sc D.~Calvetti, S.~Morigi, L.~Reichel, and F.~Sgallari}, {\em Tikhonov
  regularization and the {L}-curve for large discrete ill-posed problems}, J.
  Comput. Appl. Math., 123 (2000), pp.~423--446.

\bibitem{JulianneSilvia}
{\sc J.~Chung and S.~Gazzola}, {\em Flexible {K}rylov methods for $\ell_p$
  regularization}, SIAM J. Sci. Comput., 41 (2019), pp.~S149--S171.

\bibitem{superres}
{\sc J.~Chung, E.~Haber, and J.~Nagy}, {\em Numerical methods for coupled
  super-resolution}, Inverse Problems, 22 (2006), pp.~1261--1272.

\bibitem{Chung2015}
{\sc J.~Chung, S.~Knepper, and J.~G. Nagy}, {\em Large-Scale Inverse Problems
  in Imaging}, Springer, New York, NY, 2015, pp.~47--90.

\bibitem{JulianneJim}
{\sc J.~Chung and J.~G. Nagy}, {\em An efficient iterative approach for
  large-scale separable nonlinear inverse problems}, SIAM Journal on Scientific
  Computing, 31 (2010), pp.~4654--4674.

\bibitem{chungframe}
{\sc J.~M. Chung, M.~E. Kilmer, and D.~P. O'Leary}, {\em A framework for
  regularization via operator approximation}, SIAM Journal on Scientific
  Computing, 37 (2015), pp.~B332--B359.

\bibitem{Donatelli2006}
{\sc M.~Donatelli, C.~Estatico, A.~Martinelli, and S.~Serra-Capizzano}, {\em
  Improved image deblurring with anti-reflective boundary conditions and
  re-blurring}, Inverse problems, 22 (2006), pp.~2035--2053.

\bibitem{noiseest}
{\sc D.~Donoho}, {\em De-noising by soft-thresholding}, IEEE Trans. Inform.
  Theory, 41 (1995), pp.~613--627.

\bibitem{DykesL2021}
{\sc L.~Dykes, R.~Ramlau, L.~Reichel, K.~Soodhalter, and R.~Wagner}, {\em
  Lanczos-based fast blind deconvolution methods}, Journal of computational and
  applied mathematics, 382 (2021), p.~113067.

\bibitem{Elfving}
{\sc T.~Elfving and P.~C. Hansen}, {\em Unmatched projector/backprojector
  pairs: Perturbation and convergence analysis}, SIAM journal on scientific
  computing, 40 (2018), pp.~A573--A591.

\bibitem{inexactGKB}
{\sc S.~W. Gaaf and V.~Simoncini}, {\em Approximating the leading singular
  triplets of a large matrix function}, Applied Numerical Mathematics, 113
  (2017), pp.~26 -- 43.

\bibitem{gazzola2016inheritance}
{\sc S.~Gazzola and P.~Novati}, {\em Inheritance of the discrete picard
  condition in krylov subspace methods}, BIT Numerical Mathematics, 56 (2016),
  pp.~893--918.

\bibitem{survey}
{\sc S.~Gazzola, P.~Novati, and M.~R. Russo}, {\em {O}n {K}rylov projection
  methods and {T}ikhonov regularization}, Electron.~Trans.~Numer.~Anal., 44
  (2015), pp.~83--123.

\bibitem{newsurvey}
{\sc S.~Gazzola and M.~Sabaté~Landman}, {\em Krylov methods for inverse
  problems: Surveying classical, and introducing new, algorithmic approaches},
  Mitteilungen der Gesellschaft für Angewandte Mathematik und Mechanik, 43
  (2020).

\bibitem{GolubVP}
{\sc G.~Golub and V.~Pereyra}, {\em Separable nonlinear least squares: the
  variable projection method and its applications}, Inverse Problems, 19
  (2003), pp.~R1--R26.

\bibitem{Hanke01}
{\sc M.~Hanke}, {\em On {L}anczos based methods for the regularization of
  discrete ill-posed problems}, BIT, 41 (2001), pp.~1008--1018.

\bibitem{PCH10}
{\sc P.~C. Hansen}, {\em Discrete inverse problems: insight and algorithms},
  SIAM, Philadelphia, Pa., 2010.

\bibitem{book1}
{\sc P.~C. Hansen, J.~G. Nagy, and D.~P. O'Leary}, {\em Deblurring images :
  matrices, spectra, and filtering}, SIAM, Philadelphia, 2006.

\bibitem{HnPlSt09}
{\sc I.~Hnetynkova, M.~Plesinger, and Z.~Strakos}, {\em The regularizing effect
  of the {G}olub-{K}ahan iterative bidiagonalization and revealing the noise
  level in the data}, BIT, 49 (2009), pp.~669--696.

\bibitem{HansenMin}
{\sc T.~Jensen and P.~C. Hansen}, {\em Iterative regularization with
  minimum-residual methods}, BIT Numer. Math., 47 (2007), pp.~103--120.

\bibitem{restoreT}
{\sc J.~G. Nagy, K.~Palmer, and L.~Perrone}, {\em Iterative methods for image
  deblurring: A matlab object-oriented approach}, Numerical algorithms, 36
  (2004), pp.~73--93.

\bibitem{Ng1999}
{\sc M.~K. Ng, R.~H. Chan, and W.-C. Tang}, {\em A fast algorithm for
  deblurring models with neumann boundary conditions}, SIAM {J}ournal on
  {S}cientific {C}omputing, 21 (1999), pp.~851--866.

\bibitem{Saad2003}
{\sc Y.~Saad}, {\em Iterative Methods for Sparse Linear Systems}, Society for
  Industrial and Applied Mathematics, Philadelphia, PA, USA, 2nd~ed., 2003.

\bibitem{inexact1}
{\sc V.~Simoncini and D.~B. Szyld}, {\em Theory of inexact krylov subspace
  methods and applications to scientific computing}, SIAM Journal on Scientific
  Computing, 25 (2003), pp.~454--477.

\bibitem{inexactbis}
{\sc J.~Van Den~Eshof and G.~L.~G. Sleijpen}, {\em Inexact krylov subspace
  methods for linear systems}, SIAM journal on matrix analysis and
  applications, 26 (2004), pp.~125--153.

\bibitem{WrightOpt}
{\sc S.~Wright and J.~Nocedal}, {\em Numerical optimization}, Springer, New
  York, 1900.

\end{thebibliography}
\end{document}